\documentclass[10pt]{amsart}
\usepackage{amsmath,amssymb,amsthm,amscd}
\usepackage{graphicx}
\usepackage{color}
\numberwithin{equation}{section}
\usepackage[plainpages=false]{hyperref}
\newtheorem{prop}{Proposition}[section]
\newtheorem{thm}[prop]{Theorem}
\newtheorem{lem}[prop]{Lemma}
\newtheorem{cor}[prop]{Corollary}
\newtheorem{rem}[prop]{Remark}
\newtheorem{defi}[prop]{Definition}

\def\begeq{\begin{equation}}
\def\endeq{\end{equation}}

\begin{document}
\title{ K\"ahler-Ricci flow on  $\mathbf G$-spherical Fano manifolds}

\dedicatory{Dedicated to Professor Gang Tian on the occasion of his 65th birthday}

\author{Feng  ${\rm Wang}^\dag$ and  Xiaohua  ${\rm Zhu}^{*}$}

\address{SMS, Zhejiang
University, Hangzhou,  310028, China.}
\email{wfmath@zju.edu.cn}

\address{SMS, Peking
University, Beijing 100871, China.}
\email{xhzhu@math.pku.edu.cn}

\thanks {$\dag$ partially  supported by the National Key Research and Development Program of China No. 2022YFA1005501 and NSFC 12031017.}
\thanks {* partially supported  by National Key R\&D Program of China  2020YFA0712800 and  NSFC 12271009.}



     \begin{abstract} We prove that the  Gromov-Hausdorff  limit of K\"ahler-Ricci flow  on  a  $\mathbf G$-spherical   Fano manifold $X$
 is a  $\mathbf G$-spherical  $\mathbb Q$-Fano    variety $X_{\infty}$,  which  admits a (singular)  K\"ahler-Ricci  soliton.  Moreover, the   $\mathbf G$-spherical  variety structure of  $X_{\infty}$  can be constructed as  a  center of   torus  $\mathbb C^*$-degeneration of $X$  induced by an element   in  the Lie algebra  of  Cartan   torus   of  $\mathbf G$.

\end{abstract}

\maketitle
\tableofcontents

\section{Introduction}
 Let $(X, J, \omega_0)$ be a Fano manifold and $\omega_t$  a solution of normalized K\"ahler-Ricci (KR) flow,
   \begin{align}\label{kr-flow}
   \frac{\partial \omega_t}{\partial t}=-{\rm Ric}\,(\omega_t)+\omega_t,~\omega_0\in 2\pi c_1(M, J).
   \end{align}
   It is known that (\ref{kr-flow})  has a global solution $\omega_t$ for all $t\ge 0$ whenever the initial metric  $\omega_0$ represents $2\pi c_1(M)$ in \cite{cao}. The famous Hamilton-Tian (HT) conjecture asserts \cite{Ti97}:

    {\it Any sequence of $(X, \omega_t)$ contains a subsequence converging to a length space $(X_\infty,\omega_\infty)$
   in the Gromov-Hausdorff topology and $(X_\infty,\omega_\infty)$ is a smooth KR soliton outside a closed subset $S$, called the singular set, of codimension at least $4$. Moreover, this subsequence of $(X, \omega_t)$ converges  locally to the regular part of $(X_\infty,\omega_\infty)$
   in the Cheeger-Gromov topology.}

   This conjecture was first proved by Tian and Zhang in dimension less than $4$ \cite{TZhzh}.  Moreover, they also proved that $X_\infty$ has a structure of   $\mathbb Q$-Fano    variety.   In higher dimensions,   the conjecture  was  proved by  Bamler \cite{Bam}  and Chen-Wang \cite{Chwang},  independently.   In fact, Bamler proved a generalized version of the  HT conjecture for Ricci flow with bounded scalar curvature. 
    The authors \cite{WZ20}  recently  gave  an alternative  proof to  the HT conjecture by using   the Tian's partial $C^0$-estimate established by Zhang \cite{ZhK} for  evolved K\"ahler metrics of  KR flow.  

As an application of  HT conjecture,   Chen-Sun-Wang  proved that the limit $(X_\infty,\omega_\infty)$  is independent of  choice of subsequence  of  $\omega_t$  \cite{CSW}.   In fact, they proved  that    $X_\infty$ is a center of  $\mathbb C^*$-degeneration of another  $\mathbb Q$-Fano    variety  $\bar X_\infty$,  which is a limit of  $1$-PS degeneration of $X_\infty$  generated  by the soliton vector field (VF)  $v$ of  $X_\infty$ (also see \cite[Theorem 3.4]{DS16} and  \cite{WZ}).   Here   $v$ can be uniquely  determined by the modified  Futaki invariant introduced in  \cite{TZ2,  Xio,  BW, WZZ}.  More recently,  by showing that $\bar X_\infty$  can be regarded as  a minimizer of  H-invariant
for a class of filtrations associated to $X$ in  algebraic geometry,   Han-Li further   proved that the limit $X_\infty$   is also  independent of choice of initial metric in   the class  $2\pi c_1(X, J)$ \cite{HL}.  The H-invariant was first introduced by Tian-Zhang-Zhang-Zhu \cite{TZZZ},  and latterly generalized by  Dervan-Sz\'ekelyhidi \cite{DS16} for   $\mathbb C^*$-degenerations of $X$.

A natural question is how to determine $ X_\infty$ as well as  $\bar X_\infty$ effectively.   This  question have been  recently  investigated for Fano group compactifications and  horosymmetric  Fano manifolds \cite{LTZ1, LL,  TZ3}. In particular,  by analyzing  the singularities for  a class of  real  Monge-Amp\`ere  equations,  Tian-Zhu   \cite{TZ3}  proved that the  limit  $X_\infty$ of KR flow   on  a  $\mathbf G$-horosymmetric  Fano manifold  is a $\mathbf G$-horosymmetric  $\mathbb Q$-Fano    variety,    which  can be realized as  a  center of     $\mathbb C^*$-degeneration induced by an element in the Lie algebra  $\mathfrak t$ of Cartan   torus  $\mathbf T$ of  complex reductive Lie group $\mathbf G$ associated to $X$.

The purpose of paper is to  generalize the Tian-Zhu's result above to  $\mathbf G$-spherical   Fano manifolds.   We prove
 
\begin{thm}\label{main} Let $\mathbf G$ be a complex reductive Lie group. 
 Then the  Gromov-Hausdorff  limit of KR flow  (\ref{kr-flow}) on  a  $\mathbf G-$spherical   Fano manifold  $(X,J)$
 is a  $\mathbf G$-spherical  $\mathbb Q$-Fano    variety $X_{\infty}$,  which  admits a (singular)  KR soliton $\omega_\infty$.  Moreover, the   $\mathbf G$-spherical  variety  structure of  $X_{\infty}$ can be constructed  as  a  center of    torus  $\mathbb C^*$-degeneration of $X$ induced by an element $\tilde \Lambda$  in  $  \mathcal V_{\mathbb R}$. Here  $\mathcal V_R$ is a  valuation cone as a  subset of  the quotient space of   torus  Lie algebra  $\mathfrak t$.  

\end{thm}

\begin{rem}\label{remarkfor-v}The   element $\tilde \Lambda\in \mathcal V_{\mathbb R}$  in Theorem  \ref{main}  is constructed by  a small perturbation    $\tilde v$ of   the soliton VF $v$ of $(X_\infty, \omega_\infty)$  (cf. Lemma \ref{uni}).    $v$ can be uniquely determined    as a minimizer $\Lambda_v$  of convex functional  on   $\mathcal V_R$   which  associated to the H-invariant  restricted to  special torus   $1$-PS  degenerations (cf. Proposition \ref{ztcf} and  Proposition \ref{hinv}).   In case of horosymmetric  manifolds, it was proved  that $\tilde \Lambda$ can be chosen  as $\Lambda_v$ if  $v\neq 0$  \cite{TZ3}.     The latter implies that   both of  $\mathbb Q$-Fano varieties   $X_\infty$ and $\bar X_\infty$ are same after a projective  group transformation.  We guess that  the result is also true for  $\mathbf G$-spherical  manifolds. 
\end{rem}

By  the  torus $\mathbb C^*$-degeneration  in Theorem  \ref{main},    $X$ must be biholomorphic to $X_\infty$  if $X$ is   modified K-polystable.  Thus we prove  the following Yau-Tian-Donaldson conjecture    for  $\mathbf G-$spherical Fano  manifolds  via the method of KR flow immediately. 

\begin{cor}\label{WZ-new}
There exits a KR soliton on a $\mathbf G-$spherical Fano  manifold  if and only  it is   modified K-polystable in terms  of modified  Futaki invariant.   
%As a consequence,   $X$  is  modified K-polystable if it is modified K-semistable.
\end{cor}

The Yau-Tian-Donaldson conjecture  for KR solitons has been proved  for any Fano manifolds by  Datar-Sz\'ekelyhidi  \cite{DatS} via the continuity method and  by Chen-Sun-Wang \cite{CSW} via  the KR flow  method, respectively.  Actually, the necessary part comes from \cite{BW} (also see  \cite{DY, HL2,   BLXZ}). We would like to mention that in case of horosymmetric  manifolds the condition of   modified K-polystabilty  for the existence part  in Corollary \ref{WZ-new}  (in case of  KR solitons,  but not KE metrics) can be weakened as  the  modified K-semistabilty, see  Remark \ref{remarkfor-v}.   Thus we guess that this improvement   is still  true for    $\mathbf G-$spherical Fano  manifolds    (also see \cite[Remark 6.1-1)]{TZ3}).

By Theorem  \ref{main}, we see  that   the  limit $X_\infty$ of KR flow   (\ref{kr-flow})  
still preserves the $\mathbf G-$spherical structure  in  algebraic geometry.  On the other hand,  replacing by a class of  parabolic complex   Monge-Amp\`ere  equations for weak plurisubharmonic functions (cf. \cite{TS1, BBEGZ}),  the  KR flow    can be defined on  any $\mathbb Q$-Fano    variety with long time existence.  We hope that Theorem  \ref{main} can be generalized  to  the KR flow on a  spherical  $\mathbb Q$-Fano    variety.  In other words,  it is  very possible to establish a theory that  the   class of $\mathbf G-$spherical   $\mathbb Q$-Fano    varieties is  a closed set    in the setting of KR flows. \footnote{We are in debt to professor  Gang Tian   for pointing this out to us.}

 The class of $\mathbf G$-spherical     varieties is a very large one of highly symmetric varieties,  which contains toric varieties, generalized flag manifolds,  group compactifications and  horosymmetric  Fano manifolds  (cf. \cite{Al, Del1, Del2}). 
In  recent years,   the existence problem of (singular)  KE  metrics and (singular) KR solitons on such  $\mathbb Q$-Fano    varieties  has been extensively studied (cf. \cite{Del1, LZZ,  Del2,  LTZ2,  DH,  TZ3, LLW}, etc.).
The existence criterion  can be expressed in terms of the barycenter of moment polytope associated to the Cartan torus $\mathbf T$  of $\mathbf G$.

At last,   we  outline the  proof of Theorem  \ref{main}.    Because of lack of symmetry of  $\mathbf G$-spherical  manifolds compared with    horosymmetric   manifolds,  we will use a different approach   in  \cite{TZ3} to prove Theorem  \ref{main}.  The basic idea is to use the GIT method  via the Tian's partial $C^0$-estimate as in \cite{CSW, WZ20, WZ}.     By  the uniqueness of $X_\infty$ \cite{HL, WZ},  we need to prove  the theorem  only for the KR flow  with  $\mathbf K$-invariant initial metrics  in $2\pi c_1(X, J)$,  where  $\mathbf K$ is  a maximal compact subgroup of $\mathbf G$.
 We will  first establish a $\mathbf G_0$-equivariant  version of Chen-Sun-Wang's result  in   \cite{CSW}  for the  KR flow with  $\mathbf K_0$-invariant initial metrics  on any Fano manifold  $X$ (cf. Theorem \ref{th1}), where  $\mathbf G_0$ is  a reductive  subgroup of automorphisms  group ${\rm Aut}(X)$  of $X$ and $\mathbf K_0$ is   a maximal compact subgroup of $\mathbf G_0$.    The key point is to show that  the  $1$-PS  degeneration of $X_\infty$  induced   by the soliton VF  $v$ of  $X_\infty$ can be approximated by a  sequence of $\mathbf G_0$-equivariant   $\mathbb C^*$-degenerations (cf. Proposition \ref{RTC-appr},  Lemma \ref {commication}). Next, by using the spherical geometry for  $\mathbf G$-equivariant  $\mathbb C^*$-degenerations studied by Delcroix \cite{Del2}, we are able to show that     $X_\infty$  can be realized as  a  center  of   torus  $\mathbb C^*$-degeneration of $X$  induced by an element   in  $\mathcal V_{\mathbb R}$ (cf.  Theorem \ref{special-classification} and Lemma \ref{uni}).

The paper is organized as follows.  In Section 2,  we  use  the  notion of $1$-PS degeneration on a projective variety introduced in \cite{DT, Ti97}  to  prove  a   $\mathbf G_0$-equivariant  version of  approximation  result by a sequence of   $\mathbf G_0$-equivariant  $\mathbb C^*$-degenerations. 
% an approximation  result by a sequence of $\mathbb C^*$-degenerations as well as terms  of $\mathbf G$-equivariance
  (cf. Proposition \ref{RTC-appr}).  Here notions of $1$-PS degeneration and $\mathbb C^*$-degenerations are corresponding to  ones of  $\mathbb R$-test configuration and test configuration introduced  in \cite{DT} and \cite{Do} for the setting of schemes,  respectively. In Section 3, we  use the technique of   partial $C^0$-estimate  to apply the result in Section 2 to  the Hilbert points associated to  a family of submanifolds of   Kodaira embeddings. Theorem \ref{th1} is the main result in this section. 
In next three sections, Section 4-6, we devote  to prove Theorem  \ref{main} by  using the spherical geometry for  $\mathbf G$-equivariant   $\mathbb C^*$-degenerations as   torus    $\mathbb C^*$-degenerations described by rational piecewise linear functions  on $\mathfrak t$ in \cite{Del2} (cf. Proposition \ref{ZTCP},  \ref{mp},  \ref{qsp}).  Theorem  \ref{main}  will be finally proved in Section 6.

\vskip3mm

%  \noindent {\bf Acknowledgements.} The first author  would like to thank Yan Li for many helpful discussions on spherical varieties.

\section{ $\mathbf G-$equivariant  $1$-PS  degenerations}

 Let    $(X, L)$ be a polarized  normal variety  in $\mathbb CP^n$ such that $(\mathbb CP^n, \mathcal {O}(1) )|_{X}=L$.
   Let   $\mathbf G_0$ be  a reductive  subgroup of automorphisms  group ${\rm Aut}(X)$  of $X$ and $\mathbf K_0$  a maximal compact subgroup of $\mathbf G_0$.   We recall the definition of   $\mathbb C^*$-degenerations  for $(X, L)$ in \cite{Ti97}.  

\begin{defi}\label{tc1}
A  normal variety  $\mathcal X$   with a $\mathbb C^*$-action   is called a   $\mathbb C^*$-degeneration  of  $(X, L)$  if it satisfies:
\begin{enumerate}
\item there  is a flat $\mathbb C^*$-equivarant map $\pi: \mathcal X\to \mathbb C$ such that $\pi^{-1}(t)$ is biholomorphic to $X$ for any $t\neq 0$;

\item there is a  holomorphic line bundle $\mathcal L$ on $\mathcal X$ such that $\mathcal L|_{\pi^{-1}(t)}$ is isomorphic to $L^{r}$ for some integer $r>0$ and  any $t\neq 0$.

  \end{enumerate}

 We say that  $(\mathcal X, \mathcal L)$ 
  is $\mathbf G_0$-equivariant if it  admits an action of $\mathbf G_0$ which commutes with the $\mathbb{C}^{*}$-action such that the isomorphism between $\left(X, L \right)$ and  $(\pi^{-1}(1), ~$
  $\left.\mathcal{L}\right|_{\pi^{-1}(1)})$ is $\mathbf G_0$-equivariant.
\end{defi}

For any   $\mathbb C^*$-degeneration $(\mathcal X,\mathcal L)$,  we can compactify it by gluing a trivial family to get a family over $\mathbb P^1$. We denote the compactification by $(\bar {\mathcal X}, \bar {\mathcal L})$. Assuming that $L^r$ is very ample, we can embed $M$ into the projective space $\mathbb{P}(H^0(X,L^r)))$  by sections of $H^0(X,L^r)$. Then $(\bar {\mathcal X}, \bar {\mathcal L})$  is embedded into $\mathbb{P}(H^0(X,L^r))\times \mathbb C$ and there is a 1-PS $\tau(s)$ in  ${\rm GL}(H^0(X, L^{r}),  \mathbb C)$ such that $\mathcal X_0=\lim_{s\rightarrow 0}\tau(s)X$ and $\mathcal L$ is the pullback of $\mathcal O(1)$(cf. \cite{RT}).  Here $\mathcal X_0=\pi^{-1}(0)$ denotes the center of   $(\mathcal X,\mathcal L)$.  It also holds for $\mathbf G_0-$equivariant $\mathbb C^*$-degenerations such that  $\tau(s)$ commute with $\mathbf G_0$, where $\mathbf G_0$ can be lifted as a subgroup of ${\rm GL}(H^0(X, L^{r})$.
 In fact, we have

\begin{lem}\label{1-ps-correponding}
For any 1-PS $\tau(s)$ $(s\in \mathbb C^*)$  in ${\rm GL}(H^0(X,L^r))$ commuting with $\mathbf G_0$, there is a $G_0$-equivariant $\mathbb C^*$-degeneration $(\mathcal X, \mathcal L)$ such that $\mathcal X_0= \lim_{s\rightarrow 0}\tau(s)X$. Conversely,  any $\mathbf G_0$-equivariant $\mathbb C^*$-degeneration $(\mathcal X,\mathcal L)$ arises this way if $r$ is large enough.
\end{lem}

\begin{proof}
Given a 1-PS $\tau(s)$ in ${\rm GL}(H^0(X,L^r)$ commuting with $\mathbf G_0$, let $\mathcal X$ be the closure of $\tau(s)\cdot X$ in ${\rm P}(H^0(X,L^r))\times \mathbb C$. Since $\tau(s)$ commutes with $\mathbf G_0$, the action of $\mathbf G_0$ on ${\rm P}(H^0(X,L^r))$ preserves each $\tau(s)\cdot X$. Denote by $\mathcal L$ be the pull back of the line bundle $\mathcal O(1)$ on ${\rm P}(H^0(X,L^r))$. Then $(\mathcal X, \mathcal L)$ is a $\mathbb C^*$-degeneration of $(X,L)$.

Conversely, since $\pi$ is flat, $\pi_*\mathcal L$ is vector bundle $E$ over $\mathbb C$ when $r$ is large enough. This vector bundle is $\mathbf G_0\times \mathbb C^*$ equivariantly isomorphic to $H^0(X,L^r)\times \mathbb C$. The $\mathbb C^*$ action on the fiber $E_0\cong H^0(X,L^r)$ gives a 1-PS $\tau(s)$ which commutes with $\mathbf G_0$. The action of $\tau(s)$ is the original $\mathbb C^*$ action on $(\mathcal X,\mathcal L)$. Thus we have proved the 1-1 correspondence between 
$\mathbf G_0$-equivariant $\mathbb C^*$-degenerations and  1-PS's commuting with $\mathbf G_0$.
\end{proof}

In this paper, we will always assume that  $X$ is  a $\mathbb Q$-Fano variety and $L=K^{-r_0}$ for some integer $r_0>0$.  If the center  $\mathcal X_0$  of  $\mathbb C^*$-degeneration   $(\mathcal X,\mathcal L)$  is a $\mathbb Q$-Fano variety, 
 we say that $(\mathcal X,\mathcal L)$ is special \cite{Ti97}.

Let  $\tau(s)$ be  the  $1$-PS in ${\rm GL}(H^0(X,K_X^{-rr_0}), \mathbb C)$  associated to  $(X, rK^{-r_0}_X)$  in Lemma \ref{1-ps-correponding} such that
 $$\mathcal X_0=\lim_{s\rightarrow 0}\tau(s)X$$
   is a 
$\mathbb Q$-Fano variety. Then $\tau(s)$ induces a holomorphic VF on $\mathcal X_0$. The canonical lifting of $\tau(s)$ is another 1-PS in ${\rm GL}(H^0(M,K_X^{-rr_0}), \mathbb C)$, denoted by $\tilde\tau(s)$.  It is easy to see that  $\tau(s)=\tilde \tau(s)s^c$  for some integer $c$.  Moreover, we have  

\begin{lem}\label{anti}
$$\bar{\mathcal L}=K^{-1}_{\mathcal X/{\mathbb P}^1}+c\mathcal X_0.$$ 
\end{lem}
We can write a 1-PS $\tau(s)(s\in \mathbb C^*)$ as 
\begin{align}\tau(s)=\begin{bmatrix}
s^{k_1 } &0&...& 0  \\
0&s^{k_2 } &...& 0   \\
...&...&...&...\\
0&0 &...& s^{k_n }
\end{bmatrix}\notag
 \end{align}
where $k_i(1\leq i\leq n)$ are integers. Let $t=-\log |s|$. This matrix is equal to 
\begin{align}\tau(s)=\begin{bmatrix}
e^{-tk_1 } &0&...& 0  \\
0&e^{-tk_2 } &...& 0   \\
...&...&...&...\\
0&0 &...& e^{-tk_n }
\end{bmatrix}. \notag
 \end{align}

Now we consider general real $1$-PS $\sigma(t)$ in  ${\rm GL}(H^0(X,L^r), \mathbb R)$, where
\begin{align}\label{real-eigenvalues}\sigma(t)=\begin{bmatrix}
e^{-\lambda_1 t} &0&...& 0  \\
0&e^{-\lambda_2 t} &...& 0   \\
...&...&...&...\\
0&0 &...& e^{-\lambda_n t}
\end{bmatrix}
 \end{align}
   with respect to some basis of $H^0(X,L^r)$ and $\lambda_i \in \mathbb R$ are  the eigenvalues of $\sigma(t)$.

\begin{defi}\label{r} We call the pair $(X,  \sigma(t))$ a 1-PS  degeneration of  
 $(X, L)$. 
  For simplicity, 
  we denote it by $\mathcal X^\sigma$ and  the limit 
 % $$\mathcal X^\sigma_0$
   $$\mathcal X^\sigma_0:=\lim_{t\rightarrow \infty}\sigma(t)X$$
is called the central fiber of $\mathcal X^\sigma$.
  If the  algebraic group $\mathbf G_0$ acts on $(X,L)$ and $\sigma(t)$ commutes with the induced action of $\mathbf G_0$ on $H^0(X,L^r)$, then we say that  $\mathcal X^\sigma$ is a $\mathbf G_0$-equivariant  1-PS  degeneration. 
\end{defi}

\begin{rem}
Given any holomorphic VF $v$ on ${\rm P}(H^0(X,L^r), \mathbb C)$,   $v$ corresponds to an element in  ${\rm gl}(H^0(X,L^r), \mathbb C)$.  Suppose that $e^{t{\rm Im}(v)} $ generalizes a compact 1-PS of  ${\rm GL}(H^0(X,L^r), \mathbb C)$.   Then  $(X, e^{t{\rm Real}(v)})$ is a 1-PS  degeneration of  $(X,  L)$. This induced  $1$-PS  degeneration was first considered by Ding-Tian in \cite{DT} when the limit  of  $e^{tv} X$ is a  normal variety.
 \end{rem}
 
\begin{rem}\label{rps}If a holomorphic VF  $v$ is diagonalized as $v=\sum_{i=1}^n a_i {z_i}\frac{\partial}{\partial z_i}$, then $e^{t{\rm Im}(v)} $ generalizes a compact subgroup if and only if $a_i$ are real numbers. So we denote $e^{t{\rm Real}(v)}$ by $e^{tv}$ without confusion in the sequel.
 \end{rem}

The limit in Definition \ref{r} is the flat limit which can be defined as follows. Denote the ideal defining $X$ in ${\rm P}(H^0(X,L^r), \mathbb C)$ by $I$. For $f=\sum_\alpha c_\alpha x^\alpha$, put $$\langle \alpha,\lambda \rangle=\sum \alpha_i\lambda_i, \nu(f)=\min_{c_\alpha\neq 0}\langle \alpha,\lambda \rangle$$ and $Lt(f)=\sum_{\langle \alpha,\lambda \rangle=\nu(f)}c_\alpha x^\alpha$. Then $I_0=\{Lt(f)|f\in I\}$ is the defining ideal of $\mathcal X^\sigma_0$. It's same as 
$\lim_{t\rightarrow \infty}\sigma(t)\cdot [X]=[\mathcal X^\sigma_0]$, where $[\cdot]$ is the Hilbert point of a closed subscheme. A  $1$-PS  degeneration
can also be defined using filtrations as in \cite{DatS, HL}.

Given $n$ real numbers $\lambda_1\leq \lambda_2\leq ...\leq \lambda_n$, we say it's of type $(e_1,e_2,...,e_k)$ if $\lambda_1=\lambda_2=...=\lambda_{e_1}<\lambda_{e_1+1}=\lambda_{e_1+2}=...=\lambda_{e_1+e_2}<...$.

\begin{lem}\label{sim}
Let  
$$\sigma(t)=\begin{bmatrix}
e^{-\lambda_1 t} &0&...& 0  \\
0&e^{-\lambda_2 t} &...& 0   \\
...&...&...&...\\
0&0 &...& e^{-\lambda_n t} 
\end{bmatrix}$$  
and 
 $$\tau(t)=\begin{bmatrix}
e^{-\lambda'_1 t} &0&...& 0  \\
0&e^{-\lambda'_2 t} &...& 0   \\
...&...&...&...\\
0&0 &...& e^{-\lambda'_n t} 
\end{bmatrix}$$
   be two real $1$-PS in $\mathbb CP^{n-1}$. If $\lambda_1\leq \lambda_2\leq ...\leq \lambda_n$ and  $\lambda'_1\leq \lambda'_2\leq ...\leq \lambda'_n$ are of the same type, then for any $x=[x_1,x_2,...,x_n]\in \mathbb CP^{n-1}$, 
    $$\lim_{t\rightarrow \infty} \sigma(t)(x)=\lim_{t\rightarrow \infty} \tau(t)(x).$$
\end{lem}

\begin{proof}
Let $i_0$ be the smallest index such that $x_i\neq 0$ and $e_1+...+e_s<i_0\leq e_1+...+e_{s+1}$. Then 
$$\lim_{t\rightarrow \infty} \sigma(t)(x)=[0,...,0,x_{i_0},x_{i_0+1},...,x_{e_1+...+e_{s+1}},0,...,0]=\lim_{t\rightarrow \infty} \tau(t)(x).$$
\end{proof}

Applying the above lemma to the Hilbert point $[X]$ of $X$, we have

\begin{lem}\label{appr1}
Let
 $$\sigma(t)=\begin{bmatrix}
e^{-\lambda_1 t} &0&...& 0  \\
0&e^{-\lambda_2 t} &...& 0   \\
...&...&...&...\\
0&0 &...& e^{-\lambda_n t}
\end{bmatrix}$$
 be  a real $1$-PS as in (\ref{real-eigenvalues}).  Then for any positive number $\epsilon$, there exists a sequence of $\lambda'_j\in \mathbb Q$ such that $|\lambda'_j-\lambda_j|\leq \epsilon \,(1\leq j\leq n)$ and
  $$\lim_{t\rightarrow \infty}\sigma(t)\cdot [X]= \lim_{t\rightarrow \infty}\tau(t)\cdot [X],$$
where $\tau(t)$ is the real $1$-PS  as in   (\ref{real-eigenvalues}) with  eigenvalues $\lambda'_i$.
\end{lem}

\begin{proof}
At first, we recall the construction of the Hilbert point. For $X\subseteq \mathbb CP^{n-1}$, there is an integer number $m$ depending on the Hilbert polynomial $p$ of $X$ such that the defining ideal $I$ of $X$ is generated by homogeneous polynomials in $I$ of degree $m$. Denote the latter by $I_m$ and the whole space of homogeneous polynomials of degree $m$ by $V$. Then $\dim V=\binom{n+m-1}{n-1} $ and $\dim I_m=\binom{n+m-1}{n-1}-p(m):=k$ by the definition of Hilbert polynomial. So $I_m$ corresponds to a point in the Grassmann manifold $G_r(k,V)$. We can embed the Grassmann into $\mathbb P(\Lambda^{k}V)$ using Pl\"ucker coordinates and the corresponding point of $I_m$ is the Hilbert point of $X$. 

More precisely let $I_m=span\{h_1,h_2,...,h_{k}\}$, then $[X]$ is the point $[h_1\wedge h_2\wedge...\wedge h_{k}]\in \mathbb P(\Lambda^{k}V)$. For the standard base $x_1^{\alpha_1}x_2^{\alpha_2}...x_n^{\alpha_n} (\sum_{i=1}^n \alpha_i=m)$ of $V$, the weight of induced action for $\sigma$ is $\langle \lambda, \alpha\rangle:= \sum_{i=1}^n \lambda_i \alpha_i$. So   the weights on $\Lambda^{k}V$ are given by $\sum_{j\in I}\langle \lambda, \alpha^{(j)}\rangle$ for $|I|=k$ and $k$ different $\alpha^{(j)}$. It follows that the type $(f_1,f_2,...)$ of the induced action of $\sigma$ is determined by 
\begin{align}\label{a1}
\sum_{j\in I_1}\langle \lambda, \alpha^{(j)}\rangle=\sum_{j\in I_2}\langle \lambda, \alpha^{(j)}\rangle=...=\sum_{j\in I_{f_1}}\langle \lambda, \alpha^{(j)}\rangle<\sum_{j\in I_{f_1+1}}\langle \lambda, \alpha^{(j)}\rangle...
\end{align}
All the equations in (\ref{a1}) are of integer coefficients. So we can perturb $\lambda_i$ to rational points $\lambda_i'$ preserving these equations and inequalities, i.e., the type of the eigenvalues of the induced actions of $\sigma(t)$ and $\tau(t)$ are the same. Thus the lemma is proved by Lemma \ref{sim}.
\end{proof}

For the equivariant case, we  have

\begin{prop}\label{RTC-appr}
Assume that there is an action of  reductive  group $\mathbf G_0$ on $(X,L)$. Let $\sigma(t)$ be a $\mathbf G_0$-equivariant real 1-PS. Then there is a sequence of $\mathbf G_0$-equivariant   $1$-PS  degenerations $\tau_i(t)$ with rational eigenvalues which converges to the eigenvalues of $\sigma(t)$ such that $\mathcal X_0^{\tau_i}=\mathcal X_0^{\sigma}$.  As a consequence, there is a   $\mathbf G_0$-equivariant   $\mathbb C^*$-degeneration $\tau_i(t)$ whose central fiber is  same as $\mathcal X_0^\sigma$.
\end{prop}

\begin{proof}
$\sigma(t)$ commutes with $G_0$ if and only if the eigenspace of $\sigma(1)$ are all $\mathbf G_0$-invariant. In the proof of Lemma \ref{appr1}, we can also require that $\lambda'$ and $\lambda$ are of the same type by adding the corresponding equations and inequalities to (\ref{a1}). Then the eigenspaces of $\lambda'$ and $\lambda$ are the same. Consequently the real 1-PS $\tau(t)$ corresponding to $\lambda'$ also commutes with $\mathbf G_0$. Since $\lambda'$ can be close to $\lambda$ as small as we want, we get a sequence of $\tau_i(t)$ with rational eigenvalues which converges to the eigenvalues of $\sigma(t)$. For any $\tau_i$, there is a sufficiently divisible integer $N$ such that $\tau_i(Nt)$ defines a   $\mathbf G_0$-equivariant   $\mathbb C^*$-degeneration $(\mathcal X, \mathcal L)$ by Lemma  \ref{1-ps-correponding}. Moreover,  by Lemma \ref{sim}, 
the central fiber of $(\mathcal X, \mathcal L)$ is just   $\mathcal X_0^\sigma$.  The proposition is proved.
\end{proof}

Let $X$ be a $\mathbb Q-$Fano variety.  For a   $1$-PS  degeneration $\mathcal X^\sigma$ such that $\mathcal X^\sigma_0$ is $\mathbb Q$-Fano, $\sigma(t)$ induces a holomorphic VF $\eta(\sigma)$ on $\mathcal X^\sigma_0$.  Then by the canonical lifting of $\eta(\sigma)$ on $K_{\mathcal X^\sigma_0}^{-1}$, we get another real 1-PS $\tilde \sigma(t)$.  By Lemma \ref{anti},   $\tilde\sigma(t)$ is equal to $\sigma(t)$ up to a constant .  We say that $\tilde \sigma(t)$ is the normalized real 1-PS corresponding to $\sigma(t)$. We end this section by the following definition.

\begin{defi}
Let $X$ be a $\mathbb Q$-Fano variety.  A   $1$-PS  degeneration  $\mathcal X^\sigma$ of $X$ is called special if $\mathcal X^\sigma_0$ is $\mathbb Q$-Fano and $\sigma(t)=\tilde \sigma(t)$.
\end{defi}

\section{Equivariant K\"ahler-Ricci flow}

In this section, we 
  prove an equivariant version of Chen-Sun-Wang's result  (see an explicit statement in  \cite[Theorem 3.4]{DS16})  for the KR flow  in \cite{CSW}.   Let   $(X_\infty, \omega_\infty)$ be the  limit of KR flow  (\ref{kr-flow}) in  the HT conjecture and   $\mathbf G_0$ the   reductive subgroup of automorphisms  group ${\rm Aut}(X)$ as in Section 2.  We need to  prove 

\begin{thm}\label{th1} $X_\infty$ is a center  of  special $\mathbf {G_0}\times  \mathbb {C}^*$-equivariant  $\mathbb C^*$-degeneration of another  $\mathbb Q$-Fano    variety  $\bar X_\infty$ which is a center of  special  $\mathbf {G_0}$-equivariant  $1$-PS  degeneration of $X_\infty$  induced   by the soliton VF  $v$ of  $(X_\infty, \omega_\infty)$.
\end{thm}

\subsection{Algebraic lemmas}

As in \cite{CSW}, we use the GIT method to prove Theorem  \ref{th1}.    Let $E$ be a complex linear space and $\bar G={\rm GL}(E)$. Let $V$ be a representation space of $\bar G$. For a fixed compact subgroup $\mathbf K$  of U(N), we define a subgroup of $\bar G$ by
$$\bar G_{\mathbf K}=\{g\in \bar G|~g\text{ commutes with }\mathbf K\}.$$

Let $A_i~(i=0,1,...)$ be a sequence of matrices in $\bar G_K$ and $\Lambda$ a Hermitian matrix commuting with $\mathbf K$.  We assume that
  \begin{align}\label{p}
  \lim_{i\rightarrow \infty} A_iA_{i-1}^{-1}=e^\Lambda.
  \end{align}
  Denote the eigenvalues of $\Lambda$ by $\mathcal S=\{d_1>d_2>d_3...>d_{k-1}>d_k\}$ and the eigenspaces corresponding to $d_j$ by $U_j$. For $v\in V\setminus \{0\}$, $[v]$ is the corresponding point in the projective space $\mathbb P(V)$. Denote the set of  limit points of $A_{i_k}\cdot [v]$ for any subsequence of $i_k\to\infty$  by ${\rm Lim}_{A_i}[v]$. We will give a characterization of this set.

We fix a  U(N)-invariant metric $|\cdot|$ on $V$.   Set
 $$d(v)=\lim_{i\rightarrow \infty}(\log |A_{i+1} v|-\log |A_{i} v|).$$
   The following lemma is proved in \cite{CSW}.

  \begin{lem}\label{lim}
  Assume that  (\ref{p}) holds.  Then for $v\in V\setminus \{0\}$,
   $$d(v)=\lim_{i\rightarrow \infty}(\log |A_{i+1} v|-\log |A_{i} v|)$$
     exists and belongs to $\mathcal S$. Moreover for any $[w]\in {\rm Lim}_{A_i}[v]$, $w$ belongs to $U_{d(v)}$.
  \end{lem}

Denote $V_{j}=\{v\in V| ~d(v)\leq d_j\}$. It is obvious that $V_i,U_j$ are $\mathbf K$-invariant subspaces.  Denote $\bar G_{\mathbf K,\Lambda}=\{g\in \bar G_{\mathbf K}|~g\cdot e^{s\Lambda}=e^{s\Lambda}\cdot g\}.$ 

  \begin{lem}\label{fil}
  Assume that $V=S^m(E)$ for some $m\geq 1$. Then there exists $C\in\bar G_{\mathbf K}$, which is independent of $m$ such that $C\cdot V_j=\oplus_{i=j}^k U_i$ for all $i$.  There also exists  $C_i\in \bar G_{\mathbf K}$ with $\lim_{i\rightarrow \infty} C_i=Id$ such that
  $\tilde A_i=C_iA_iC^{-1} \in \bar G_{\mathbf K,\Lambda}$.
  \end{lem}
  
\begin{proof}
For two $\mathbf K-$invariant subspaces with same dimension,  we can always choose the isomorphism between them to be $\mathbf K-$invariant. Then the lemma can be proved as in \cite[p3153]{CSW}.
\end{proof}

  \begin{prop}\label{alg}
  Assume that $V=S^m(E)$ for some $m\geq 1$. Define $[\bar v]=\lim_{i\rightarrow \infty} e^{i\Lambda} C\cdot[ v]$. Then ${\rm Lim}_{A_i}[v]\subset \overline{ \bar G_{\mathbf K,\Lambda} \cdot [\bar v]}$.
  \end{prop}
  \begin{proof}
  
 Assuming that $d(v)=d_j$, we have $v\in V_j$. So $C\cdot v\in \oplus_{i=j}^k U_i$, and the $U_j$ component $\pi_j(C\cdot v)$ of $C\cdot v$ is not zero. It follows that $[\bar v]=[\pi_j(C\cdot v)]$. For $[w]\in {\rm Lim}_{A_i}[v]$, by Lemma \ref{lim}, we know that $w\in U_j$. Assume that $[w]=\lim_{i\rightarrow \infty} [A_{\alpha_i}v]$, then we have
  \begin{align}
  \lim_{i\rightarrow \infty} \tilde A_{\alpha_i} [\bar v]&=
  \lim_{i\rightarrow \infty} \tilde A_{\alpha_i} [\pi_j(C\cdot v)]
  =\pi_j(\lim_{i\rightarrow \infty} \tilde A_{\alpha_i} [(C\cdot v)]) \notag \\
  &=\pi_j(\lim_{i\rightarrow \infty} A_{\alpha_i}[v])=\pi_j[w]=[w].  \notag
  \end{align}
  The proposition is proved since $\tilde A_i \in \bar G_{\mathbf K,\Lambda}$.
  \end{proof}

\subsection{Proof of Theorem \ref{th1}} 

By  the uniqueness of $X_\infty$ \cite{HL, WZ},  
 we may assume that  the initial metric $\omega_0$ in  (\ref{kr-flow})  is  a $\mathbf K_0$-invariant metric.   Then all  solutions of  $\omega_t$  are  also $\mathbf K_0$-invariant.  Choose the multilple $l$ large enough such that the partial $C^0$-estimate holds for $(M,\omega_t)$ by \cite{ZhK, WZ20}. Let $h_t$ be the $\mathbf K_0$-invariant Hermitian metric on $K_X^{-1}$ such that the Chern curvature of $h_t$ is $\omega_t$. It induces a $K-$invariant metric $H_t$ on 
$E=H^0(X,K_X^{-l})$ by
\begin{align}\label{inner-product}
 ( s^\alpha, s^\beta)_{H_t}=\int_M\langle   s^\alpha, s^\beta \rangle_{h_t^{\otimes l}}\omega_{t}^n.
 \end{align}
Let $\{s^\alpha(t)\}$   be an  ortho-normal  basis  of   $(E,H_t)$ and $\Phi_t: M\to \mathbb CP^N$ be the Kodaira embedding given by  $\{s^\alpha(t)\}$.   By  Proposition 2.7 in \cite{WZ20},  $ \Phi_{t_i}(X)$  converges to a $\mathbb Q$-Fano variety $X_\infty$ with klt singularities for a subsequence $t_i\rightarrow \infty$. Moreover by Theorem 1.1 in \cite{WZ20}, $X_\infty$ admits a weak K\"ahler-Ricci soliton. 

We can make these embeddings equivariant with respect to $\mathbf K_0$ by using suitable ortho-normal basis in the following. Choose an ortho-normal basis $s^\alpha(0)$ with respect to  $H_0$.  Then    with respect to $\{s^\alpha(0)\}$ we can write $H_t$ as a matrix. By abuse of notations, we also denote it by $H_t$.

Consider  the following ODE for  matrix $A_t$:
 \begin{align}\label{ODE-s}
\left\{ 
\begin{aligned}
  \frac{d}{dt}A_t&=-\frac{1}{2}A_tH_t'H_t^{-1},\\
 A_0&=id. 
\end{aligned}\right.
  \end{align} 
   It follows that $\frac{d}{dt}(A_tH_tA_t^*)=0$ and $A_tH_tA_t^*=id$.
Define $s^\alpha(t)=\sum_{\beta=0}^N A^{\alpha\beta}_ts^\beta(0)$. Then  one can check that $\{s^\alpha(t)\}$ is an ortho-normal basis of $E$ with metric $H_t$. Denote by $\Phi_t$ the embeddings  of $X$ given by
   $\{s^\alpha(t)\}$ as above. Thus  we have $\Phi_t=A_t\cdot \Phi_0$. 
   
   By (\ref{inner-product}),  the inner  $(\cdot,\cdot)_{H_0}$ on the basis $\{s^\alpha(0)\}$ induces $\mathbf K_0$ as a  subgroup of  U(N+1). Then $\mathbf G_0$ can be regarded as a subgroup of GL(N+1, $\mathbb C$).

\begin{lem}\label{commication}
$A_t$ commutes with $\mathbf K_0$ and so it commutes with $\mathbf G_0$ .
\end{lem}

\begin{proof}
Since $H_t(g\cdot,g\cdot)=H_t$ for $g\in \mathbf K_0$, we have $g H_t g^*=H_t$ which  is same as $H_tg=gH_t$ since $g$ is unitary. So $H_t'$ also commutes with $\mathbf K_0$. Now by (\ref{ODE-s}) we have 
$$\frac{d}{dt}[A_t, g]=-\frac{1}{2}[A_t,g]H_t'H^{-1}_t,$$ 
where $[A,B]=AB-BA$. Hence $[A_t, g]=0$.  By the reductivity of  $\mathbf G_0$, we also get
$[A_t, g]=0$ for any $g\in \mathbf G_0$.  The lemma is proved.
\end{proof}

\begin{proof}[Completion of proof of Theorem \ref{th1}] As we know that  the ortho-normal basis $\{s^\alpha(0)\}$ induces the  action of $\mathbf G_0$  as a subgroup of  GL(N+1, $\mathbb C$), i.e., $\Phi_0$ commutes with $\mathbf G_0$.  By  the above lemma,  $\Phi_t$ also commutes with $\mathbf G_0$.  On the other hand,  by Lemma 6.5 in \cite{WZ},  $A_{t_i'+s}\cdot A_{t_i'}^{-1}$ converges to $e^{sv}$ uniformly for any subsequence $t_i'\rightarrow \infty$, where $v$ is the soliton vector VF  on the limit of K\"ahler-Ricci flow. Since $v$ is the limit of $A_{t_i'+s}\cdot A_{t_i'}^{-1}$, it commutes with $\mathbf K_0$. By Corollary 3.3 in \cite{WZ}, $v$ is unique up to conjugation. We may assume $v\in gl(N+1, \mathbb C)$ is diagonal. By \cite[Proposition 7.4]{WZ}, ${\rm Im}(v)$ lies in the compact part of $Aut^v(X_\infty)$, where $X_\infty$ is the limit of K\"ahler Ricci flow. It follows that the eigenvalues of $v$ are all real numbers. So we can modify $A_i$ by unitary elements commuting with $\mathbf K_0$ such that the condition (\ref{p}) holds with $\Lambda=v$. Denote the Hilbert polynomial of $\Phi_0(X)$ by $p$ and the Hilbert scheme of subschemes of $\mathbb P(E)$ with Hilbert polynomial $p$ by $\mathbf {Hilb}^{\mathbb P(E),p}$. Then we can apply Proposition \ref{alg} to  the Hilbert point $[X]=[\Phi_0(X)]\in \mathbf {Hilb}^{\mathbb P(E),p}$ as in \cite{WZ}. Denote $[\bar X_\infty]=\lim_{t\rightarrow \infty} e^{tv} C\cdot [X]$. Then by Propsition \ref{alg} $[X_\infty]$ lies in the closure of $G_{\mathbf {K_0},v}\cdot [\bar X_\infty]$.

Note that  both of  $C$ and $e^{tv}$ commute with $\mathbf K_0$.
Thus 
$$\lim_{t\rightarrow \infty} e^{tv} C\cdot X$$
is the center  of $\mathbf G_0$-equivariant $1$-PS $\sigma(t)=e^{tv}$  degeneration of $(X, K_X^{-1})$.   Clearly,  the Hilbert point of  this limit is just $[\bar X_\infty]$.  Hence,  we  may denote the limit by $ \bar X_\infty$. On the other  hand, 
  by \cite[Proposition 7.4]{WZ},  ${\rm Aut}^v(X_\infty)$ is reductive, so $G_{\mathbf K_0,v}\bigcap{\rm Aut}^v(X_\infty)$ is also a reductive group.  So by Luna's slice lemma, there is a $1$-PS $\tau(s)$ $(s\in \mathbf C^*)$   in  $G_{\mathbf K_0,v}\bigcap{\rm Aut}^v(X_\infty)$ degenerating $\bar X_\infty$ to $X_\infty$. Since all the elements in $G_{\mathbf K_0,v}$ commute with $\mathbf G_0\times \mathbb C^*$,  by Lemma \ref{1-ps-correponding}, $\tau(s)$  induces a  $\mathbf G_0\times \mathbb C^*$-equivariant  $\mathbb C^*$-degeneration of  $(\bar X_\infty, K_{X_\infty}^{-l})$ with the center $X_\infty$.  The theorem is proved.
\end{proof}

\section{Torus $1$-PS  degenerations on $\mathbf G$-spherical varieties}

In this section, we will recall Delcroix's work \cite{Del-2020}  on  torus  $\mathbb C^*$-degenerations of spherical varieties and then  we generalize his results for torus $1$-PS   degenerations.

\subsection{Combinatorial data associated to  $\mathbf G$-spherical varieties }

 We  list the general facts about spherical varieties and give some basic results of combinatorial data.   We  refer the reader to   the main references   \cite{K, Tim}  in this subsection. 
Let $\mathbf G$ be a complex connected reductive group and $\mathbf B \subseteq \mathbf G$ be a Borel subgroup.

\begin{defi}\label{sph}
The homogeneous variety 
$\mathbf G / \mathbf H$ is said spherical if there is an open $\mathbf B$-orbit   $\mathbf B \mathbf H$
 and the corresponding orbit map  $\pi: \mathbf G \rightarrow \mathbf G / \mathbf H$ is separable. The $G$-action on $G/H$ is on the left: $\sigma\cdot g\mathbf H=(\sigma g)\mathbf H$.  We denote by $\mathbf P\subseteq \mathbf G$ the stabilizer of the open $\mathbf B$-orbit $\mathbf B \mathbf H$.

An embedding of  $\mathbf G / \mathbf H$ is a normal $\mathbf G$-variety $X$ together with a $\mathbf G$-equivariant open embedding  $\mathbf G / \mathbf H \hookrightarrow X$. Let $x_0$ be the image of $\mathbf H\in \mathbf G / \mathbf H$. We say $(X, x_0)$ or $X$ is a spherical variety.

\end{defi}

The group of characters of any group $\mathbf S$ is denoted by $\mathfrak X(\mathbf S)$:
$$\mathfrak X(\mathbf S)=\{f: S\rightarrow \mathbb C^*\text{ is a homomorphism}\}.$$
  Note that for a Borel subgroup $\mathbf B$ and the maximal subtorus $\mathbf T$ contained in $\mathbf B$, $\mathfrak X(\mathbf T)=\mathfrak X(\mathbf B)$.  Let $\Phi \subset \mathfrak{X}(\mathbf T)$ be the root system of $(\mathbf G, \mathbf T)$ and denote by $\Phi^{+}$ the positive roots in $\Phi$ determined by $\mathbf B$.

The parabolic subgroup $\mathbf P$ of $\mathbf G$ containing $\mathbf B$ admits a unique Levi subgroup $\mathbf L$ containing $\mathbf T$. Let $\Phi_{\mathbf P}$ denote the set of roots of $\mathbf P$ with respect to $\mathbf T$. We denote by $\Phi_{\mathbf P^{u}}$ the set of roots of the unipotent radical $\mathbf  P^{u}$ of $\mathbf P$. Define
$$
\kappa_{\mathbf P}=\sum_{\alpha \in \Phi_{\mathbf P^{u}}} \alpha .
$$

For a normal variety $X$, denote by $\mathbb C(X)$ the rational functions on $X$. A valuation of a normal variety $X$ is a map
$$
v: \mathbb C(X)^{*}=\mathbb C(X) \backslash\{0\} \longrightarrow \mathbb{Q}
$$
with the properties
$$
\begin{aligned}
&v\left(f_{1}+f_{2}\right) \geq \min \left\{v\left(f_{1}\right), v\left(f_{2}\right)\right\} \text { whenever } f_{1}, f_{2}, f_{1}+f_{2} \in \mathbb C(X)^{*}, \\
&v\left(f_{1} f_{2}\right)=v\left(f_{1}\right)+v\left(f_{2}\right) \text { for all } f_{1}, f_{2} \in \mathbb C(X)^{*} \text { and } \\
&v\left(\mathbb C^{*}\right)=0 .
\end{aligned}
$$
Define
\begin{align}
(\mathfrak M:=)\mathfrak M(\mathbf G/\mathbf H)=&\{\chi \in \mathfrak X(\mathbf B)| \mathbb C( \mathbf G/\mathbf H)^{(\mathbf B)}_\chi\not=\{0\}\}, \notag \\
(\mathfrak N:=)\mathfrak N( \mathbf G/\mathbf H)=&{\rm Hom}(\mathfrak M( \mathbf G/\mathbf H)),\mathbb Z). \notag
\end{align}

Then $\mathfrak M$ is a subgroup of $\mathfrak X(\mathbf B)=\mathfrak X(\mathbf T)$ and its rank  is called the rank of homogeneous space $ \mathbf G/\mathbf H)$.

For any valuations $v$ on $\mathbb C( \mathbf G/\mathbf H)$, define $\rho(v)=\rho_v\in \mathfrak N\otimes \mathbb Q$ by $\rho_v(\chi)=v(f_\chi)$. Define $(\mathcal V:=)\mathcal V( \mathbf G/\mathbf H)=\{\mathbf G\text{ invariant valuations}\}.$ Then it is known that $\rho: \mathcal V\rightarrow  \mathfrak N_\mathbb Q$ is injective. So we can identify $\mathcal V( \mathbf G/\mathbf H)$ with its image and call it the valuation cone. The valuation cone $\mathcal V$ spans $ \mathfrak N_\mathbb Q$. Moreover  $\mathcal{V}\subseteq \mathfrak{N}_{\mathbb{Q}}$ is a cosimplicial cone (cf. \cite[Section 20]{Tim}). Hence there is a uniquely determined linearly independent set $\Sigma \subseteq \mathfrak{M}$ of primitive elements such that
$$
\mathcal{V}=\bigcap_{\gamma \in \Sigma}\left\{v \in \mathfrak N\otimes \mathbb Q:\langle v, \gamma\rangle \leq 0\right\}
$$
The elements of $\Sigma$ are called the spherical roots which is a subset of cone spaned by positive roots of $\mathbf G$.

Denote the set of $B$-stable prime divisors in $ \mathbf G/\mathbf H$ by $\mathcal D:=\mathcal D( \mathbf G/\mathbf H)$, which are called colors of $ \mathbf G/\mathbf H$. We also denote the set of $B$-stable prime divisors in $X$ by $\mathcal P(X)$. Let $X$ be an embedding with $\mathbf G$-closed orbit $Y$. Define
$\mathcal D_Y(X)=\{D\in \mathcal P(X), Y\subseteq D\}$.

The embedding $X$ is called simple if there is only one closed $G$ orbit $Y$. In this case, we define two sets
\begin{align*}
\mathcal I_Y(X)&=\{v_D\in \mathcal V| D\in \mathcal D_Y(X)\text{ is $G$-stable}\},\\
\mathcal J_Y(X)&=\{\emptyset \neq D\bigcap  \mathbf G/\mathbf H\in \mathcal D| D\in \mathcal D_Y(X)\}.
\end{align*}
Then $X$ is uniquely determined by $\mathcal I_Y(X), \mathcal J_Y(X)$. Now we want to describe this pair only using information of $ \mathbf G/\mathbf H$. For any color $D$, $v_D$ is a valuation, then we denote $\rho(v_D)$ simply by $\rho(D)$. Let $\mathcal C_Y(X)\subseteq  \mathfrak N_\mathbb Q$ be the cone generated by $\rho(\mathcal J_Y(X))$ and $\mathcal I_Y(X)$.
\begin{lem}
For a simple embedding $X$, the $\mathbb Q_+$ span of $\mathcal I_Y(X)$ are exactly the extremal rays of $\mathcal C_Y(X)$ which contain no elements of $\rho(\mathcal J_Y(X))$.
\end{lem}

\begin{defi}
A colored cone is a pair $(\mathcal{C}, \mathcal{R})$ with $\mathcal{C} \subseteq \mathfrak N_\mathbb Q$ and $\mathcal{R} \subseteq \mathcal{D}$ having the following properties:
\begin{itemize}
\item CC1: $\mathcal{C}$ is a cone generated by $\rho(\mathcal{R})$ and finitely many elements of $\mathcal V$.
\item CC2: The interior $\mathcal{C}^{\circ} \cap \mathcal V\neq \emptyset$
\end{itemize}

A colored cone is called strictly convex if the following holds:\\
SCC: $\mathcal{C}$ is strictly convex and $0 \notin \rho(\mathcal{R})$.
\end{defi}
Denote the pair of $(\mathcal C_Y(X),\mathcal J_Y(X))$ by $\mathcal{C}^{c}(X)$, then we have

\begin{thm}
The map $X \mapsto \mathcal{C}^{c}(X)$ is a bijection between isomorphism classes of simple embeddings and strictly convex colored cones.
\end{thm}

For the general spherical embedding of $\mathbf G/\mathbf H$, every $\mathbf G$ orbit gives a colored cone and these colored cones constitute a colored fan.

\begin{defi}
A colored fan is nonempty finite set $\mathcal F$ of colored cones with the following properties:
\begin{itemize}
\item CF1: Every face of $\mathcal{C}^{c} \in \mathcal{F}$ belongs to $\mathcal{F}$.
\item CF2: For every $v \in \mathcal{V}$ there is at most one $(\mathcal{C}, \mathcal{R}) \in \mathcal{F}$ with $v \in \mathcal{C}^\circ$.
\end{itemize}
A colored fan $\mathcal{F}$ is called strictly convex if $(0, \varnothing) \in \mathcal{F}$, or equivalently, if all elements of $\mathcal{F}$ are strictly convex.
\end{defi}

\begin{thm}
The map $X \mapsto \mathcal{F}(X)$ induces a bijection between isomorphism classes of embeddings and strictly convex colored fans.
\end{thm}

Denote the union of all sets $\mathcal R$ with $(\mathcal C,\mathcal R)\in \mathcal F(X)$ by $\mathcal D(X)$, which are called colors  of $X$.

Let $\mathcal{I}_{X}^{G}$ denote the finite set of $G$-stable prime divisors of $X$. Any divisor $Y \in \mathcal{I}_{X}^{G}$ corresponds to a ray $(\mathcal{C}, \emptyset) \in \mathcal{F}(X)$, and we denote by $u_{Y}$ the indivisible generator of this ray in $\mathfrak{N}$. For any $D\in \mathcal D$, its closure $\bar D$ in $X$ is an irreducible $B$-stable but not $G$-stable divisor. Any $B$-stable Weil divisor $d$ on $X$ writes
\begin{align}\label{B-divisor}
d=\sum_{Y \in \mathcal{I}_{X}^{G}} n_{Y} Y+\sum_{D \in \mathcal{D}} n_{D} \bar{D}
\end{align}
for some integers $n_{Y}, n_{D}$. In fact, any Weil divisor is linearly equivalent to a $B$ invariant divisor and Brion proved the following criterion to characterize Cartier divisors.

\begin{prop}\cite[Proposition 3.1]{Bri89} A $B$-stable Weil divisor $d$ in $X$ of form \eqref{B-divisor} is Cartier if and only if there exists an integral piecewise linear function $l_{d}$ on the support $|\mathcal{F}(X)|$ of the fan $\mathcal{F}(X)$ such that
\begin{align}\label{div}
d=\sum_{Y \in \mathcal{I}_{X}^{G}} l_{d}\left(u_{Y}\right) Y+\sum_{D \in \mathcal{D}(X)} l_{d}(\rho(D)) \bar{D}+\sum_{D \in \mathcal{D} \backslash \mathcal{D}(X)} n_{D} \bar{D}
\end{align}
for some integers $n_{D}$.
\end{prop}

Let $\mathcal{F}^{\max}(X)$ denote the set of cones $\mathcal{C} \subseteq$ $\mathfrak{N}_\mathbb{Q}$ of maximal dimension, such that there exists $\mathcal{R} \subseteq \mathcal{D}$ with $(\mathcal{C}, \mathcal{R}) \in \mathcal{F}(X)$. If $d$ is a Cartier divisor and $\mathcal{C} \in \mathcal{F}^{\max}(X)$, let $v_{\mathcal{C}}\in\mathfrak{M}$ such that $l_{d}(x)=v_{\mathcal{C}}(x)$ for $x \in \mathcal{C}$. Moreover, $d$ is ample if and only if the following conditions are satisfied:
\begin{itemize}
\item the function $l_{d}$ is convex,

\item $v_{\mathcal{C}_{1}} \neq v_{\mathcal{C}_{2}}$ if $\mathcal{C}_{1} \neq \mathcal{C}_{2} \in \mathcal{F}^{\max}(X)$,

\item $-n_{D}>v_{\mathcal{C}}(\rho(D))$ for all $D \in \mathcal{D} \backslash \mathcal{D}(X)$ and $\mathcal{C}\in \mathcal{F}^{\max}(X)$.
\end{itemize}

To a Cartier divisor $d$, we associate a polytope $\Delta_{d} \subseteq \mathfrak{M}_\mathbb{R}$ as
\begin{align*}
\Delta_d:=\left(\cap_{\mathcal C\in \mathcal{F}(X)^{\max}}(-v_{\mathcal{C}}+\sigma^{\vee})\right)\bigcap\left(\cap_{D \in \mathcal{D} \setminus \mathcal{D}_{X}}\{m|m(\rho(D))+n_{D} \geq 0\}\right).
\end{align*}
If $d$ is ample then $l_{d}$ is the support function of $\Delta_d$ for $x \in\left|\mathcal{F}(X)\right|$.

There is also a moment polytope associated to an ample line bundle. Let $L$ be a $G$-linearized ample line bundle on a spherical variety $X$. Denote by $V_{\lambda}$ an irreducible representation of $G$ of highest weight $\lambda \in \mathfrak{X}(T)$ with respect to $B$. Since $X$ is spherical, for all $r \in \mathbb{N}$, there exists a finite set $\Delta_{r} \subset \mathfrak{X}(T)$ such that
$$
H^{0}\left(X, L^{r}\right)=\bigoplus_{\lambda \in \Delta_{r}} V_{\lambda}
$$
\begin{defi} The moment polytope $\Delta_{L}$ of $L$ with respect to $B$ is defined as the closure of $\bigcup_{r \in \mathbb{N}^{*}} \Delta_{r} / r$ in $\mathfrak{X}(T) \otimes \mathbb{R}$.
\end{defi}

The relationship between the moment polytope and polytope of a $B$-stable divisor is as follows. Fix a linearization of $L$, and choose a global $B$-semi-invariant section $s$ of $L$, so that the zero divisor $d$ of $s$ is an ample $B$-invariant Cartier divisor. Let $\mu_{s}$ be the character of $B$ defined by $s$, that is, such that $b \cdot s\left(b^{-1} \cdot x\right)=\mu_{s}(b) s(x)$ for all $x \in X$.

\begin{prop}\label{pl}(\cite{Bri89}, Proposition 3.3])
The moment polytope $\Delta_{ L}$ of $ L$ and the polytope $\Delta_{d}$ associated to $d$ are related by $\Delta_{L}=$ $\mu_{s}+\Delta_{d}$.
\end{prop}

\subsection{Torus   $\mathbb C^*$-degenerations  for $\mathbf G$-polarized spherical varieties}

The following classification result is proved in \cite{Del-2020} based on works in  \cite{Do, AB2,AK}. We include its proof from \cite[Section 4]{Del-2020} below for our later use.

\begin{prop}\label{ZTCP}\cite[Theorem 4.1]{Del-2020}
Let $(X,L)$  a polarized $\mathbf G$-spherical variety. Let $d$ be any fixed $\mathbf B$-stable divisor of $L$. Then the $\mathbf G$-equivariant    $\mathbf C^*$-degenerations  of $(X,L)$ are in one-to-one correspondence with positive rational piecewise linear concave functions on $\Delta_d$, with slopes in $\mathcal V$.
\end{prop}

\begin{proof}
Let $(\mathcal{X}, \mathcal{L})$ be a $\mathbf G$-equivariant test configuration for $(X, L)$. Recall the compactification of $(\mathcal{X}, \mathcal{L})$ . That is, we glue the trivial family over $\mathbb{C}$ to $(\mathcal{X}, \mathcal{L})$ along $\mathbb{C}^{*}$ to obtain a family over $\mathbb{P}^{1}$. We denote the point added to $\mathbb{C}$ by $\infty$ and still denote by $(\mathcal{X}, \mathcal{L})$ the compactified family.

{Up to adding a sufficiently large multiple of the pull back of $\{0\}\in\mathbb P^1$, we may assume that $\mathcal L$ is ample.} Then $(\mathcal{X}, \mathcal{L})$ is a polarized spherical variety under the action of $G \times \mathbb{C}^{*}$. Its open orbit is $G / H \times \mathbb{C}^{*}$, and the combinatorial data are easily derived from that of $X$.

{Suppose that $s$ is the defining section of $d$.} Let $\hat{s}$ be the $\mathbb{C}^{*}$-invariant meromorphic section of $\mathcal{L}$ whose restriction to $\left(X, L^{r}\right)=$ $\left(\mathcal X_{1}, \mathcal{L}_{1}\right)$ coincides with $s^{\otimes r}$. Then the divisor $\hat d$ of $\hat{s}$ is $\left(B \times \mathbb{C}^{*}\right)$-stable, hence an integral linear combination of the form
\begin{align}\label{bdiv}
\hat d=\sum_{\hat{D} \in \mathcal{P}_{\mathcal{X}}} n_{\hat{D}} \hat{D}
\end{align}
where $\mathcal{P}_{\mathcal{X}}$ is the set of prime $\left(B \times \mathbb{C}^{*}\right)$-stable divisors on $\mathcal{X}$.
There are three types of divisors in $\mathcal{P}_{\mathcal{X}}$:
\begin{itemize}
\item[(1)] The divisors $\{\hat{D} \in \mathcal{P}_{\mathcal{X}}|\hat{\varrho}(\hat{D}) \in(\mathfrak N_\mathbb{Q} \times\{0\})\}$. Any such $\hat D$ satisfies $\hat{D}=\overline{D \times \mathbb{C}^{*}}$ for some $D \in \mathcal{P}_{X}$. Since $\hat s|_{\mathcal X_1}\cong rL$, $n_{\hat{D}}=r n_{D}$. In particular, any color of $G/H\times \mathbb C^*$ must be of this type; %Note that since the Borel subgroup of $\mathbb{C}^{*}$ is $\mathbb{C}^{*}$ itself, all other elements of $\mathcal{P}_{\mathcal{X}}$ must be $\left(G \times \mathbb{C}^{*}\right)$-stable.
\end{itemize}
The remaining two types consist of the following $\left(G \times \mathbb{C}^{*}\right)$-stable divisors of $\mathcal X$:
\begin{itemize}
\item[(2)] The the fiber $X_{\infty}$ at $\infty\in\mathbb P^1$. Since $\hat s|_{\mathcal X\setminus\mathcal X_0}\cong s\times(\mathbb P^1\setminus\{0\})$, $n_{\mathcal X_0}=0$;

\item[(3)] The reduced, irreducible components $\mathcal X_{0,1},...,\mathcal X_{0,n_0}$ of $\mathcal X_0$. In this case $\hat\rho(\mathcal X_{0,i})=(u_i, -m_i)$ is a primitive vector in $\mathfrak N\oplus\mathbb Z$ with $u_i \in \mathcal{V}$ and $m_i>0$. Let us write also $n_i$ for the corresponding coefficient in (\ref{bdiv}).
\end{itemize}
Let $\hat{\Delta}$ denote the polytope in $\hat{\mathcal M}_\mathbb{R}$ associated to the divisor (\ref{bdiv}). In view of the previous description of the divisor, the polytope can be described as
$$
\hat{\Delta}=\{(r y, t) \mid y \in \Delta, 0 \leq t \leq g(x)\}
$$
where $g$ is a (positive) rational concave piecewise linear function on $\Delta$, expressed as
\begin{align}\label{pl-function}
g(x)=\inf _{i=1,...,n_0}\left(\frac{r u_i(y)+n_i}{m_i}\right) .
\end{align}
Note that since each $m_i$ is positive, each $\frac r{m_i}  u_i\in\mathcal{V}$.

Conversely, let $g$ be a positive rational piecewise linear concave function on $\Delta_d$ with slopes in the valuation cone of $X$. Then we may assume that $g>0$ is of form \eqref{pl-function} for some $r\in \mathbb N_+$, each $n_i\in\mathbb Z$, $u_i\in\mathcal V\cap\mathfrak N$ and $m_i\in\mathfrak N_+$ so that $(u_i,-m_i)$ is a primitive vector in $\mathfrak N\oplus\mathbb Z$. Consider the polytope
$$
\hat{\Delta}=\{(r x, t) \mid x \in \Delta_d, 0 \leq t \leq g(p)\},
$$
where $d$ is the $B$-semiinvariant divisor of the line bundle $L$ as in (\ref{div}).

Note that the colors of $\mathbf G / \mathbf H \times \mathbb{C}^{*}$ can be identified with colors of $ \mathbf G / \mathbf H$. Indeed, for any color $\hat D\in\hat{\mathcal D}( \mathbf G / \mathbf H\times \mathbb C^*)$, $\overline{\hat D}\cap G/H\in\mathcal D( \mathbf G / \mathbf H)$. Conversely, for every $D\in\mathcal D(G/H)$, $\overline{D\times\mathbb C^*}\in\hat{\mathcal D}( \mathbf G / \mathbf H \times \mathbb C^*)$. We are going to define a colored fan consisting of the following maximal colored cones:
\begin{itemize}
\item[(1')] $(\mathcal C \times\{0\}, \mathcal R)$, where $(\mathcal C, \mathcal R)$ is a colored cone of $\mathcal{F}(X)$;
\item[(2')] $(\mathcal C \times \mathbb{R}_{+}, \mathcal R)$, where $(\mathcal C, \mathcal R)$ is a colored cone of $\mathcal{F}(X)$;
\item[(3')] The inner normal cone of $\hat \Delta$ at a vertex $v$ of the graph of $g$. In this case, we define the colored cone $(\mathcal C_v, \mathcal R_v)$, where $\mathcal C_v$ is the inner normal cone at $v$ and 
    $$\mathcal R_v=\{\hat D=\overline{D\times C^*}|D\in {\mathcal D}( \mathbf G / \mathbf H)~\text{satisfies}~n_{D}-\rho(D)(v)=0\}.$$ 
    Here in the last equality $\rho(D)$ is considered as an vector in $\mathfrak N_\mathbb Q\times\{0\}\subset\hat{\mathfrak N}_\mathbb Q$.

\end{itemize}

The polytope $\hat{\Delta}$ is the polytope associated to the $\left(B \times \mathbb{C}^{*}\right)$-stable Cartier divisor
$$
\hat d=\sum_{\hat{D} \in \mathcal{P}_{X}} n_{\hat{D}} \hat{D}
$$
where
\begin{itemize}
\item[(1")] $n_{\hat{D}}=r n_{D}$ for each divisor $\hat{D} \in \mathcal{P}_{\mathcal{X}}$ so that $\hat{D}=\overline{D \times \mathbb{C}^{*}}$ for some $D \in \mathcal{P}_{X}$;
\item[(2")] $n_{X_{\infty}}=0$ for the only divisor $X_{\infty}$;
\item[(3")] $n_{\hat{D}}=n_{i}$ for the $G\times\mathbb C^*$-stable prime divisor $\hat D_i$ corresponding to $(u_i, -m_i)$. 
\end{itemize}
In particular, the restriction of $\hat d$ to $\mathcal X \setminus\{\hat D_i|i-1,...,n_0\}$ is the product of a divisor $d$ of form \eqref{B-divisor} in $X$ with $\mathbb{C}^{*} \cup\{\infty\}$. Furthermore, $d$ satisfies the ampleness criterion. The associated line bundle $\mathcal{O}_{\mathcal X}(\hat d)$ is $(G \times \mathbb{C}^{*})$-linearizable (up to passing to a sufficiently divisible tensor power). Choosing the linearization such that the natural section $\hat s$ of $\mathcal{O}_{\mathcal X}(\hat d)$ is $\mathbb{C}^{*}$-invariant and has $B$-weight $r\chi$. It follows $\mathcal{O}_{\mathcal X}(\hat d)|_{X \times(\mathbb{C}^{*} \cup\{\infty\})}\cong(X,rL)\times(\mathbb{C}^{*} \cup\{\infty\})$. Hence $(\mathcal{X}, \mathcal{L})$ is a test configuration of $(X,L)$.

\end{proof}

According to Proposition \ref{ZTCP}, a  $\mathbf G$-equivariant    $\mathbb C^*$-degeneration on  a $\mathbf G$-spherical variety is determined by a  rational piecewise linear function.   As    $\mathbb C^*$-degenerations on toric manifolds introduced  by Donaldson \cite{Do},   we  will use the notion  of  torus   $\mathbb C^*$-degenerations for such  degenerations for simplicity.  

If $\tau(s)$ is a $1$-PS in ${\rm GL}(H^0(X,L^k))$ commuting with $\mathbf G$, then the eigenspaces of $\tau(s)$ are $\mathbf G$-invariant subspaces of $H^0(X,L^k)$. Fix a $\mathbf B$-semiinvariant divisor of $L^k$ of form \eqref{B-divisor} with associated $\mathbf B$-character $\mu_d$. The eigenvalues of $\tau(s)$ is determined by the eigenvalues $s^{(k)}_\lambda$ on $V_{\lambda+k\mu_d}$ for $\lambda \in k\Delta_d\cap\mathfrak M(\mathbf G/\mathbf H)$.

\begin{lem}\label{ncf}
Let $(X,L)$ be a polarized projective variety and $\mathcal X$ a  torus $\mathbb C^*$-degeneration of $(X,L)$.   Then the central fiber of $\mathcal X$ is normal if and only if it corresponds to an integral affine function
\begin{align}\label{f-affine}
f(y)=C+\Lambda(y),~y\in \Delta_d
\end{align}
with $\Lambda\in\mathcal V\cap\mathfrak N$. Moreover, the eigenvalues of the 1-PS subgroup in ${\mathbb P}(H^0(X,L^k))$ corresponding to $\mathcal X$ is
\begin{align}\label{ZTC-point-dis-special}
s^{(k)}_\lambda=kf(\frac\lambda k),~\forall \lambda\in k\Delta_d\cap\mathfrak M~.
\end{align}
\end{lem}

\begin{proof}
Note that the total space $(\mathcal X,\mathcal L)$ is a polarized spherical embedding of $ \mathbf G/\mathbf H \times\mathbb C^*$ and the central fibre $\mathcal X_0$ is a $G\times \mathbb C^*$-invariant divisor. By \cite[Theorem 15.20]{Tim}, it suffices to show that $\mathcal X_0$ is prime.

$\mathcal X$ corresponds to a concave, rational piecewise linear function
$$f(y)=\min_{a=1,...,N_f}\{C_a+\Lambda_a(y)\},~y\in \Delta_d$$
with each $C_a\in\mathbb Q$ and $\Lambda\in\mathcal V\cap\mathfrak N$. Then
$$\mathcal X_0=\sum_{a=1}^{N_f}m_a\hat D_a,$$
where $\hat D_a$ is the prime $ \mathbf G/\mathbf H\times\mathbb C^*$-divisor corresponding to the $a$-th piece of $f$ and $m_a$ is the smallest positive integer so that $m_a\Lambda_a\in\mathfrak N$. Clearly $\mathcal X_0$ is prime if $N_f=1$ and $m_1=1$. Thus we get \eqref{f-affine}. \eqref{ZTC-point-dis-special} follows from the $\mathbb C^*$-action in the construction in Proposition \ref{ZTCP}.
\end{proof}

In the case of   torus $\mathbb C^*$-degeneration with reduced center fiber in Lemma \ref{ncf},   the central fiber is a spherical embedding of $ \mathbf G/\mathbf H_0$ for some $\mathbf H_0$ by Theorem 3.30 in \cite{Del1}{\color{red}}. $\mathbf H_0$ can be computed as follows: Choose $x_0\in X$ such that the isotropy group of $x_0$ is $H$ and take $\hat x_0=(x_0,1).$ Put $\bar x_0=\lim_{t\rightarrow 0}(\Lambda(t),t)\hat x_0$. Then $\mathbf H_0$ is the isotropy group of $\bar x_0$. Since the isotropy group of $(\Lambda(t),t^m)\hat x_0$ is $\mathbf H_t={\rm Ad}_{(\Lambda(t),t)}\mathbf H$ and the limit of $\mathbf H_t$ is a subgroup of $\mathbf H_0$, it must hold by dimension that
$$\mathbf H_0=\lim_{t\rightarrow 0}{\rm Ad}_{(\Lambda(t),t)}\mathbf H.$$

Now we give a  classification of   torus   $1$-PS degenerations associated to  $\mathbf G-$equivariant  $1$-PS $\sigma(t)$  in ${\rm GL}(H^0(X,L^k))$.  For simplicity,  we also denote the eigenvalues of $\sigma(t)$ on $V_\mu$ by $s^{(k)}_\mu$. Denote by $\mathcal V_{\mathbb R}$ the closure of $\mathcal V$ in $\mathcal V\otimes \mathbb R$.

\begin{prop}\label{mp}
Suppose that   the center  ${\mathcal X}_0^\sigma$  of  torus   $1$-PS  degeneration $(X, \sigma(t))$  is normal.   Then there exists $\Lambda \in \mathcal V_{\mathbb R}$ and a constant $C$ such that $s^{(k)}_\mu=\Lambda(\mu)+C, \forall \mu \in \Delta_k$. The converse is also true.
\end{prop}

\begin{proof}
By Proposition \ref{RTC-appr}, we can find a sequence of $G$-equvariant real 1-PS $\tau^i(t)$ with eigenvalues $s^{(k,i)}_\mu\in \mathbb Q$ converges to $s^{(k)}_\mu$ such that the central fiber $\mathcal X^{\tau^i(t)}_0$ is the same as $\mathcal X^{\sigma}_0$. For a sufficently divisible $N$, $\tilde \tau^i(t)=\tau^i(Nt)$ is a 1-PS. Applying Lemma \ref{ncf} to $\tilde \tau^i(t)$, there is $\tilde \Lambda_i\in \mathcal V$ and constant $\tilde C_i$ such that $Ns^{(k,i)}_\mu=\tilde \Lambda_i(\mu)+\tilde C_i$. Therefore $s^{(k,i)}_\mu=\Lambda_i(\mu)+C_i$ for some $\Lambda_i\in \mathcal V$ and constant $C_i$.
Since $s^{(k,i)}_\mu$ converges to $s^{(k)}_\mu$, by subsequence $\Lambda_i$ and $C_i$ converge to $\Lambda\in \mathcal V_{\mathbb R}$ and $C\in \mathbb R$. Then $s^{(k)}_\mu=\Lambda(\mu)+C$ and the Proposition is proved.

We can also prove this as follows. In the decomposition $H^0(X,L^k)=\oplus_{\mu\in \Delta_k} V_\mu$, choose sections $x_{\mu}\in V_\mu$ with highest weight $\mu$, i.e., 
$$b\cdot x_\mu =\mu(b)x_{\mu}, \forall b\in B.$$ For any equality such as $\sum_{i\in I}\mu_i=\sum_{j\in J}\mu_j$ with $|I|=|J|$, there is a constant $c$ such that $\Pi_{i\in I}x_{\mu_i}=c\Pi_{j\in J}x_{\mu_j}$ is a defining equation of $X$ in $\mathbb P(H^0(X,L^k))$. The weight of the two sides of the equation are $\sum_{i\in I}s^{(k)}_{\mu_i}$ and $\sum_{j\in J}s^{(k)}_{\mu_j}$. If the two weights are not the same, there will be a defining equation $\Pi_{i\in I}x_{\mu_i}=0$ or $\Pi_{j\in J}x_{\mu_j}=0$ on the limit ${\mathcal X}_0^\sigma$. It contradicts with the normality. So we must have 
$$\sum_{i\in I}s^{(k)}_{\mu_i}=\sum_{j\in J}s^{(k)}_{\mu_j}.$$ If follows that there is a affine function $\Lambda(\mu)+C$ such that $s^{(k)}_\mu=\Lambda(\mu)+C$. We are going to show that $\Lambda \in \mathcal V$. Consider the induced action of $\sigma(t)$ on $H^0(X,L^{rk})$ for any $r\geq 1$. The eigenvalues $s^{(rk)}_\mu(\mu \in \Delta_{kr})$ of the induced action is also equal to $\Lambda_r(\mu)+C_r$ for some $\Lambda_r$ and $C_r$. We have the decompostion $$V_{\mu_1}\cdot V_{\mu_2}...\cdot V_{\mu_r}=\oplus_\alpha V_{\sum_{i=1}^r \mu_i-\alpha},$$ where $\alpha$ is some non-negative linear combination of simple roots of $G$. The eigenvalue of $\sigma(t)$ on $V_{\sum_{i=1}^r \mu_i}$ is equal to $\sum_{i=1}^r s^{(k)}_{\mu_i}$. So we get $\sum_{i=1}^r (\Lambda({\mu_i})+C)=\Lambda_r({\sum_{i=1}^r \mu_i})+C_r$. It follow that $\Lambda _r=\Lambda, C_r=rC$. On the other component $V_{\sum_{i=1}^r \mu_i-\alpha}$, the eigenvalue of $\sigma(t)$ is at least $\sum_{i=1}^r s^{(k)}_{\mu_i}$. So $\Lambda(-\alpha
)\geq 0$. Since $\mathcal V$ is minus the dual cone of all such $\alpha's$, we know that $\Lambda\in \mathcal V_\mathbb R$.

Conversely, given any $\Lambda \in \mathcal V_{\mathbb R}$ and any constant $C$, the eigenvalues $s^{(k)}_\mu=\Lambda(\mu)+C$ define a $G$-equivariant real 1-PS. Denote it by $\sigma(t)$. We have to show that the central fiber is normal. The equations determines the type of the eigenvalues of induced action of $\sigma(t)$ on the Hilbert point have the following form:
 $$\sum_{i\in I} \Lambda(u_i)=\sum_{j\in J} \Lambda(u_j), |I|=|J|.$$ So we can choose a sequence of $\Lambda_i\in \mathcal V$ which are solutions of these equations and converges to $\Lambda$. Denote the real 1-PS with eigenvalues $\Lambda_i(\mu)+C$ by $\tau^i(t)$. Then the types of the eigenvalues of induced action of $\tau^i(t)$ and $\sigma(t)$ are the same. So we have $\mathcal X_0^{\tau^i(t)}=\mathcal X_0^{\sigma(t)}$. By Lemma \ref{ncf} $\mathcal X_0^{\tau^i(t)}$ is normal. Hence $\mathcal X_0^{\sigma(t)}$ is normal and the proposition is proved.  
\end{proof}

\section{H-invariant  on  $\mathbf G$-spherical Fano varieties}

\subsection{Special  torus $1$-PS  degenerations}

For a spherical variety $X$, $K^{-1}_X$ is a Weil divisor, which can be written as

\begin{align}\label{candiv}
d=\sum_{Y \in \mathcal{I}_{X}^{G}} Y+\sum_{D \in \mathcal{D}} n_{D} \bar{D}
\end{align}
where the $n_{D}$ are explicitly obtained in terms of $\kappa_{P}$  and the types of the roots (see \cite{GH15} for a precise description of these coefficients). If this is an ample $\mathbb Q$-Cartier divisor, then $X$ is a $\mathbb Q$-Fano spherical variety. In this case, there is a $B$-semi-invariant section $s$ of $K^{-1}_X$ of weight $\kappa_{P}$ such that $div(s)=d$.

The moment polytope $\Delta_{+}:=\Delta_{K_X^{-1}}$ is then $\kappa_{P}+\Delta_{d}$ by Proposition \ref{pl}, and furthermore the dual polytope $\Delta_{d}^{*}$ of $\Delta_{d}$ is a $\mathbb{Q}$-$G / H$-reflexive polytope in the sense of \cite{GH15}, which can be obtained as the convex hull:
$$
\Delta_{d}^{*}=\operatorname{conv}\left(\left\{\rho_{D} / n_{D}, D \in \mathcal{D}\right\} \cup\left\{u_{Y}, Y \in \mathcal{I}_{X}^{G}\right\}\right) .
$$
The $\mathbb{Q}$-Fano variety $X$ can further be recovered from its $\mathbb{Q}$-$G / H$-reflexive polytope $\Delta_{d}^{*}$ by the following procedure, detailed in \cite{GH15}. The colored fan of $X$ is obtained from $\Delta^*_{d}$ as the union of the colored cones $({\rm Cone}(F), \rho^{-1}(F))$ for all faces $F$ of $\Delta_{d}^{*}$ such that the intersection of the relative interior of $\operatorname{Cone}(F)$ with the valuation cone $\mathcal{V}$ is not empty.

By Lemma \ref{ncf}, the special  torus $\mathbb C^*$-degeneration of a spherical Fano variety corresponds to a linear function on $\Delta_+=\Delta_d+\kappa_P$. More precisely, the constant term can be determined by the following lemma.

\begin{prop}\label{qsp}
The special   torus    $\mathbb C^*$-degeneration  on a $\mathbb Q$-Fano   spherical variety  corresponds to  $f(y)=\Lambda y$ for $\Lambda \in \mathcal V\bigcap \mathfrak N$ and $y\in\Delta_d$.
\end{prop}

\begin{proof}
For  a  special   torus    $\mathbb C^*$-degeneration,  we have $\mathcal L=K^{-1}_{\mathcal X/\mathbb C}$. Assume that this  degeneration  corresponds to the function $f(y)=\Lambda(y)+C$.  By the proof of Proposition \ref{ZTCP} we can write the divisor $div(\hat s)$ as $\hat d+C\mathcal X_0$. Since the restriction  of the divisor $div(\hat s)$ on
generic fiber is the canonical divisor, we know that $\hat d$ is the divisor of $K^{-1}_{\mathcal X/\mathbb C}$. So we have $C=0$.
\end{proof}

Like Proposition \ref{mp}, we have the following theorem for special   torus    $1$-PS degenerations  of  spherical Fano manifold.

\begin{thm}\label{special-classification}
Assume that $H^0(X,K_X^{-k})$ is very ample. For any $\Lambda \in \mathcal V_\mathbb R$, the real 1-PS with eigenvalues
\begin{align}\label{s-k-lambda-special}
s_\lambda^{(k)}=\Lambda(\lambda),~\forall \lambda\in k\Delta_d\cap \mathfrak M
\end{align}
defines a  special torus    $1$-PS degeneration  $\mathcal X^\sigma$ of $(X,K^{-1}_X)$. Conversely, any special torus    $1$-PS degeneration  of $(X,K^{-1}_X)$ is given in this way.
\end{thm}

Theorem \ref{special-classification} is a combination of the following  two propositions.

\begin{prop}\label{special-classification-1}
For any $\Lambda \in \mathcal V_\mathbb R$, the real 1-PS $\sigma(t)$ with eigenvalues given by \eqref{s-k-lambda-special} defines a  $\mathcal X^\sigma$ of $(X,K^{-1}_X)$.
\end{prop}

\begin{proof}
Eq. \eqref{s-k-lambda-special} defines     a  special torus    $1$-PS degeneration associated to  $\sigma(t)$.  By the proof of Proposition \ref{mp}, we can choose a sequence $\Lambda_i\in \mathcal V$ so that $\Lambda_i\to\Lambda$, and $\mathcal X^{\tau^i}_0=\mathcal X^{\sigma}_0$ where $\tau^i(t)$ is the real 1-PS with eigenvalues $\Lambda_i(\lambda)$. By Proposition \ref{qsp} each $\mathcal X_{\tau^i}$ is special. So $\mathcal X^{\sigma}_0$ is $\mathbb Q$-Fano. Denote the normalized real 1-PS corresponding to $\sigma(t)$ by $\tilde\sigma(t)$. Since the actions of $\tau^i(t)$ on $\mathcal X^{\sigma}_0$ also converges to the action of $\sigma(t)$, we have $\tilde \sigma(t)=\lim_{i\rightarrow \infty} \tau^i(t)=\sigma(t).$ We conclude that $\mathcal X^\sigma$ is special.
\end{proof}

Now we prove the inverse direction:

\begin{prop}\label{special-classification-2}
For any  special torus  1-ps degeneration  of $(X,K^{-1}_X)$,  there is a $\Lambda \in  \mathcal V_\mathbb R$ so that the real 1-PS with eigenvalues \eqref{s-k-lambda-special} corresponds to it.
\end{prop}

\begin{proof}
Denote the eigenvalues of  $\sigma(t)$ on $V_\lambda$ by $s^{(k)}_\lambda$. By Proposition \ref{mp}, we can choose a sequence of $\mathbf G$-equivariant 1-PS $\tau^i(t)$ that converges to $\sigma(t)$ such that the central fiber $\mathcal X^{\tau^i(t)}_0$ is the same as $\mathcal X^{\sigma}_0$. Denote the normalized real 1-PS corresponding to $\tau^i(t)$ by $\tilde\tau^i(t)$. So $\mathcal X^{\tilde \tau^i}$ is special. By Proposition \ref{qsp} there is a $\Lambda_i\in\mathcal V$ so that the eigenvalues of $\tilde \tau^i(t)$ is  $\Lambda_i(\lambda)$ for $\lambda\in k\Delta_d\bigcap {\mathfrak M}$.
Considering the action on the central fiber $\mathcal X^\sigma_0$,
we know that 
$$\lim_{i\rightarrow\infty}\tilde \tau^i(t)= \sigma(t).$$
 Thus $\Lambda_i$ converges to some $\Lambda\in\mathcal V_{\mathbb R}$ such that
$$s_{\lambda}^{(k)}=\Lambda(\lambda),~\forall\lambda\in k\Delta_d\cap{\mathfrak M}.$$
\end{proof}

\subsection{A formula of H-invariant}

In this subsection,  we compute the  H-invariant for    special torus  $1$-PS  degenerations on a  $\mathbf G-$spherical   Fano manifold.  A similar formula has  also been obtained by Li-Li by using filtration geometry \cite{LL2}. Here we give a simple proof by using  the equivariant Riemann-Roch theorem.

As  in \cite{TZZZ},  we introduce  the definition of H-invariant for special $1$-PS degenerations as follows.

\begin{defi}
For a special   $1$-PS degeneration $\mathcal X^\sigma$, $\sigma(t)$ induces a holomorphic VF $\eta(\sigma)$ on $\mathcal X^\sigma_0$. We define $$V=\int_{ \mathcal X^\sigma_0}\omega_0^n,~ H(\mathcal X^\sigma):=V\log(\frac{1}{V}\int_{ \mathcal X^\sigma_0}e^{\theta}\omega_0^n),$$
where $\omega_0$ is the restriction of Fubini-Study metric and $\theta$ is the potential function corresponding to the canonical lifting of $\eta(\sigma)$.
\end{defi}

In general, the $H$-invariant is defined by
$$H(\mathcal X^\sigma)=-\int_{\mathcal X^\sigma_0} \tilde \theta e^{h_0} \omega_0^n+V\log\left(\frac{1}{V} \int_{\mathcal X^\sigma_0}e^{\tilde\theta}\omega_0^n\right),$$
where $\tilde\theta$ is any hamiltonian for vector field $\eta(\sigma)$ and $h_0$ is the Ricci potential of $\omega_0$ normalized by $\int_{\mathcal X^\sigma_0} e^{h_0}\omega_0^n=V$.
For the canonical lifting, it holds that
$$\Delta \theta+\theta+\langle \nabla \theta, \nabla h_0 \rangle=0.$$ If follows  that $\int_{\mathcal X^\sigma_0}  \theta e^{h_0} \omega_0^n=0$ and the two definitions coincide. By \cite{DS16} the first step degeneration from $X$ to $\bar X_\infty$ is charecterized by the minimum of the $H$-invariant.

We also need the definition of modified K-stability (cf. \cite{Xio, BW, WZZ}).  Let $X$ be  a Fano manifold and $v$  a holomorphic VF on $X$. Denote by $T$ the torus generated by $v$.

\begin{defi}
For a special   $T$-equivariant  $\mathbb C^*$-degeneration  $(\mathcal{X}, \mathcal{L})$, denote by $\theta$ and $\theta_v$  the potential functions corresponding to the cannonical lifing of $\mathbb C^*$-action and $v$ on $\mathcal X_0$ respectively. Put
$$Fut_v(\mathcal{X}, \mathcal{L})=\int_{\mathcal X_0} \theta e^{\theta_v}\omega_0^n.$$
If $Fut_v(\mathcal{X}, \mathcal{L})\geq 0$ for any special  $T$-equivariant  $\mathbb C^*$-degeneration  $(\mathcal{X}, \mathcal{L})$, we say that $(X,v)$ is modified K-semistable. Moreover, if the equality holds only when $\mathcal X\cong X\times \mathbb C$ equivariantly, we say that $(X,v)$ is modified K-polystable.
\end{defi}

If $(X,v)$ is modified K-semistable, then $v$ is determined by the modified Futaki invariant as in \cite{TZ2}.  It has  proved  in  \cite{BW,  DY, HL2,   BLXZ}  that if  $X$ admits KR soliton  with respect to $v$,   then $(X,v)$ is modified K-polystable.

By using the equivariant Riemann-Roch theorem (cf. \cite{WZZ}),   we  prove

\begin{prop}\label{ztcf}
Let  $\mathcal X$ be  a     special  torus $1$-PS degeneration induced by $\Lambda\in \mathcal V$.  Then 
$$H(\mathcal X)=V\ln\frac1V\int_{\Delta_+}e^{\Lambda(\kappa_P-y)}\pi\,dy,$$ where $\pi(y)=\frac{\prod_{\alpha\in\Phi_+,\alpha\not\perp\Delta_+}\langle\alpha,y\rangle}{\prod_{\alpha\in\Phi_+,\alpha\not\perp\Delta_+}\langle\alpha,\rho\rangle}$.

\end{prop}

\begin{proof}
Note that $\Delta_+=\Delta_d+\kappa_P$. By Lemma \ref{ncf} and Proposition \ref{qsp}, the weight of action of 1-PS $\tau(t)$ corresponding to $\Lambda$ on $V_\lambda\subseteq H^0(X, K_X^{-k})$ is $\Lambda(\lambda-k\kappa_P)$.  Denote the eigenvalues of $\tau(t)$ on $H^0(X,K_X^{-k})$ by $\lambda_i^{(k)}$.   By  \cite[Section 2]{WZZ},  we have
$$\sum e^{t\frac{\lambda_i^{(k)}}{k}}=k^n\int_{\mathcal X_0}e^{-t\theta_0}\frac{\omega_0^n}{n!}+O(k^{n-1}),$$
where $\theta_0$ is the potential function corresponding to the canonical lifting of $\tau(t)$.

From the decomposition of $H^0(X,L^k)$ and Weyl's formula of dimension, we have
\begin{align}
\sum e^{t\frac{\lambda_i^{(k)}}{k}}&=\sum_{\lambda\in \Delta_k}\dim V_\lambda e^{t\frac{\Lambda(\lambda-k\kappa_P)}{k}}\notag\\ &=\sum_{\lambda\in \Delta_k} \frac{\prod_{\alpha\in\Phi_+,\alpha\not\perp\Delta_+}\langle\alpha,\rho+\lambda\rangle}{\prod_{\alpha\in\Phi_+,\alpha\not\perp\Delta_+}\langle\alpha,\rho\rangle}e^{t\frac{\Lambda(\lambda-k\kappa_P)}{k}}
\end{align}
Writing $\lambda=ky$ we have
$$\frac{1}{k^n}\sum e^{t\frac{\lambda_i^{(k)}}{k}}=\sum_{y\in \frac{1}{k}\Delta_k} \frac{1}{k^{dim \Delta_+}}\frac{\prod_{\alpha\in\Phi_+,\alpha\not\perp\Delta_+}\langle\alpha,\frac{\rho}{k}+y\rangle}{\prod_{\alpha\in\Phi_+,\alpha\not\perp\Delta_+}\langle\alpha,\rho\rangle}e^{t\Lambda(y-\kappa_P)}.$$
Letting $k\rightarrow \infty$ we get
$$\int_{\mathcal X_0}e^{-t\theta_0}\frac{\omega_0^n}{n!}=\int_{\Delta_+}\frac{\prod_{\alpha\in\Phi_+,\alpha\not\perp\Delta_+}\langle\alpha,y\rangle}{\prod_{\alpha\in\Phi_+,\alpha\not\perp\Delta_+}\langle\alpha,\rho\rangle}e^{t\Lambda(y-\kappa_P)}dy.$$
The proposition is proved by putting $t=-1$.

\end{proof}

\begin{prop}\label{hinv}
Given any real 1-PS $\sigma(t)$ corresponding to $\Lambda\in \mathcal V_\mathbb R$, we have
\begin{align}\label{Hfor}
H(\mathcal X^\sigma)=V\ln\frac1V\int_{\Delta_+}e^{\Lambda(\kappa_P-y)}\pi\,dy.
\end{align}
As a consequence, $H$ has a unique minimizer on $\mathcal V_\mathbb R$.
\end{prop}

\begin{proof}
For the special $1$-PS degeneration $\mathcal X^\sigma$, we have a sequence of  special $\mathbb C^*$-degenerations $\mathcal X_i$ with the same central fiber $\mathcal X_0$. Assuming that $\mathcal X_i$ corresponds to $\Lambda_i \in \mathcal V$, we have $\Lambda_i\rightarrow \Lambda$. Denote the holomorphic vector fileds on $\mathcal X_0$ induced from $\mathcal X$ and $\mathcal X_i$ by $v$ and $v_i$. Then $v_i\rightarrow v$ and $\theta_{v_{i}} \rightarrow \theta_v$. So  $H(\mathcal X_i)\rightarrow H(\mathcal X^\sigma)$. By Proposition \ref{ztcf}, the proposition is proved. From (\ref{Hfor}), $H$ is a convex function of $\Lambda$. Hence there is a unique minimizer of $H$.
\end{proof}

If $v$ is a holomorphic VF commutting with the $\mathbf G-$action, then $v$ corresponds to an element $\Lambda_0$ in the linear part of $\mathcal V_\mathbb R$. Using the same method in Proposition \ref{hinv}, we have

\begin{prop}\label{modf}
Given any real 1-PS $\sigma(t)$ corresponding to $\Lambda\in \mathcal V_\mathbb R$, we have
\begin{align}\label{Ffor}
Fut_v(\mathcal X^\sigma)=\int_{\Delta_+}\Lambda(\kappa_P-y)e^{\Lambda_0(\kappa_P-y)}\pi\,dy,
\end{align}
where $\Lambda_0$ is the element in $\mathcal V_\mathbb R$ corresponding to $v$.
\end{prop}

\begin{cor}\label{semi}
If $(X,v)$ is modified K-semistable, then $H$ invariant attains its minimum at $\Lambda_0$, where $\Lambda_0$  is the element corresponding to $v$.
\end{cor}

\begin{proof}
At first $v$ is the holomorphic VF  determined by the modified Futaki invariant, so it commutes with $G$ by \cite{TZ2}. By Proposition \ref{hinv}, it only suffices to prove
\begin{align}\label{e1}
 \frac{d}{dt}_{t=0^+} H(\Lambda_0+t\Lambda)\geq 0, \forall \Lambda \in \mathcal V_\mathbb R
\end{align}
and 
\begin{align}\label{e2}
\frac{d}{dt}_{t=0} H(\Lambda_0+t\Lambda)= 0, \text{ if } \Lambda_0+t\Lambda \in \mathcal V_\mathbb R \text{ for small } t.
\end{align}
By (\ref{Hfor}) and (\ref{Ffor}), we have 
\begin{align}
& \frac{d}{dt}_{t=0} H(\Lambda_0+t\Lambda)\notag\\
&=\frac{V}{\int_{\Delta_+}e^{\Lambda_0(\kappa_P-y)}\pi\,dy}\int_{\Delta_+}\Lambda(\kappa_P-y)e^{\Lambda_0(\kappa_P-y)}\pi\,dy  \notag \\
&=\frac{V}{\int_{\Delta_+}e^{\Lambda_0(\kappa_P-y)}\pi\,dy}Fut_v(\mathcal X^\sigma), \notag
\end{align}
where $\sigma$ is the real 1-PS  corresponding to $\Lambda$. So (\ref{e1}) is proved by the assumption.
If $\Lambda_0+t\Lambda \in \mathcal V_\mathbb R$ for small $t$, denote by $\tau(t)$ the real 1-PS  corresponding to $\Lambda_0-t_0\Lambda$ for a fixed small positive numbr $t_0$. Then we have
\begin{align}
 &\frac{d}{dt}_{t=0} H(\Lambda_0+t\Lambda)\notag\\
 &=\frac{V}{\int_{\Delta_+}e^{\Lambda_0(\kappa_P-y)}\pi\,dy}\int_{\Delta_+}\Lambda(\kappa_P-y)e^{\Lambda_0(\kappa_P-y)}\pi\,dy  \notag \\
&=-\frac{V}{t_0\int_{\Delta_+}e^{\Lambda_0(\kappa_P-y)}\pi\,dy}\int_{\Delta_+}(\Lambda_0-t_0\Lambda)(\kappa_P-y)e^{\Lambda_0(\kappa_P-y)}\pi\,dy  \notag \\
&=-\frac{V}{t_0\int_{\Delta_+}e^{\Lambda_0(\kappa_P-y)}\pi\,dy}Fut_v(\mathcal X^\tau)\leq 0, \notag
\end{align}
since $Fut_v(v)=0$.
Combined with (\ref{e1}), we get (\ref{e2}). The corollary is proved.
\end{proof}

\section{Proof of Theorem \ref{main}}

In case of    $\mathbf G$-spherical manifold $X$,   we choose $\mathbf G_0=\mathbf G=\mathbf K^{\mathbb C}$ in Section 3.   Then we can consider solution  $\omega_t$ of (\ref{kr-flow}) starting with a $\mathbf K$-invariant metric $\omega_0$.  We will investigate the limit of KR flow on $X$ in more detail.  Using the same notations  in section 3,  we denote $E=H^0(X, K_X^{-l})$, 
$\bar G={\rm GL}(E)$ and 
$$\bar G_{\mathbf K}=\{g|~g \text{ commutes with } \mathbf K \}.$$
  Our key observation comes from  the following lemma.

\begin{lem}
$\bar G_{\mathbf K}$ is a torus.
\end{lem}

\begin{proof}The proof of the lemma essentially uses this property of spherical manifolds. 
For a polarized manifold $(X,L)$ with a $\mathbf G$-action, by Theorem 25.1 in \cite{Tim} $X$ is spherical if and only if every irreducible representation in the decomposition of $H^0(X,L^l)$ has multiplicity one for any $l\geq 1$.  Then we have 
$$E=H^0(X, K_X^{-l})=\oplus_{\mu\in\Delta_l}V_\mu.$$
  For any $g\in \bar G_{\mathbf K}$, it must hold that $g(V_\mu)=V_\mu$ since different $\mu's$ correspond to different irreducible representations of $\mathbf G$. By Schur's lemma, the restriction of $g$ onto each $V_\mu$ is a scalar matrix. So $g$ can be written as $diag\{ a_\mu \}$. It follows that $\bar G_{\mathbf K}$ is a torus.
\end{proof}

\begin{lem}\label{uni}
There is a perturbation $\tilde v$ of the soliton VF  $v$ such that $X_\infty$ is biholomorphic to $\lim_{i\rightarrow \infty} e^{i \tilde v}\cdot X$. Moreover, there is an element $\tilde \Lambda \in \mathcal V_{\mathbb R}$ corresponding to this degeneration.
\end{lem}

\begin{proof}
As in (\ref{ODE-s}), we have defined matrices $A_t$. By Lemma \ref{commication}, $A_t$ commutes with $\mathbf K$. So from the the lemma above, we can write $A_i=diag\{a_i^1,...,a_i^N\}\in G_{\mathbf K}$ and $v=\{diag\{\lambda^1,...,\lambda^N\}\in Lie(G_{\mathbf K})$. Then for $1\leq k\leq N$, $\lim_{i\rightarrow \infty} a_i^k (a_{i-1}^k)^{-1}=e^{\lambda^k}$. It follows that $a_i^k=e^{(\lambda^k+\epsilon_i^k)i}$ with $\lim_{i\rightarrow \infty} \epsilon_i^k=0$. Denote the Hilbert point $[X]$ by $[z_\alpha]$ and the vector $\lambda=(\lambda^k), \epsilon_i=(\epsilon_i^k)$.  As in Lemma \ref{appr1}, the induced action of $A_i$ on $[X]$ will be $[e^{-i\langle w_\alpha, \lambda+\epsilon_i \rangle}z_\alpha]$ for some vector $w_\alpha$ with integral components. Let $Q=\{u\in Lie(G_{\mathbf K})|\langle w_\alpha, u\rangle=\langle w_\beta, u\rangle, \forall \alpha\neq \beta\}$. For any $\lambda \in Lie(G_{\mathbf K})$, $e^\lambda \cdot [X]=e^{pr_{Q^\perp}(\lambda)} \cdot [X]$, where $pr_{Q^\perp}$ means the projection to the complement of $Q$. Now we replace $\lambda, \epsilon_i$ by their projections and keep the notations for simplicity. Denote $c=min\{\langle w_\alpha,\lambda\rangle|z_\alpha\neq 0\}$ and $I=\{\alpha|\langle w_\alpha,\lambda\rangle=c, z_\alpha\neq 0\}$.  $I$ is a not the whole set. Then we have
\begin{align}
[\bar X_\infty]=\lim_{i\rightarrow \infty} e^{v i}\cdot [X]&=\lim_{i\rightarrow \infty}[e^{-i\langle w_\alpha, \lambda\rangle}z_\alpha]=[z_\alpha(\alpha \in I),0].
\end{align}

Choose $\alpha_0\in I$, we have 
$$\lim_{i\rightarrow \infty}[e^{-i\langle w_\alpha, \lambda+\epsilon_i \rangle}z_\alpha]=\lim_{i\rightarrow \infty}[e^{-i\langle w_\alpha-w_{\alpha_0}, \lambda+\epsilon_i \rangle}z_\alpha].$$
Then for $\alpha \notin I$, $\lim_{i\rightarrow \infty} e^{-i\langle w_\alpha-w_{\alpha_0}, \lambda+\epsilon_i \rangle}=0.$  But for $\alpha \in I$,  it holds 
 $$\langle w_\alpha-w_{\alpha_0}, \lambda+\epsilon_i \rangle=\langle w_\alpha-w_{\alpha_0}, \epsilon_i \rangle.$$
  Thus  we get 
\begin{align}\label{l1}
\lim_{i\rightarrow \infty}[e^{-i\langle w_\alpha, \lambda+\epsilon_i \rangle}z_\alpha]=\lim_{i\rightarrow \infty}[e^{-\langle w_\alpha, i\epsilon_i \rangle}z_\alpha(\alpha\in I),0].
\end{align}

If the norms of $i\epsilon_i$ is bounded, we can assume that $\lim_{j\rightarrow \infty} i_j\epsilon_{i_j}=\epsilon \in Lie(G_\mathbf K)$ by subsequence. Then the right hand side of (\ref{l1}) is $[e^{-\langle w_\alpha, \epsilon \rangle}z_\alpha(\alpha\in I),0]$, which is equal to $e^\epsilon \cdot[\bar X_\infty]$. If the norms of $i\epsilon_i$ goes into infinity, we can assume that
$\lim_{j\rightarrow \infty}\frac{i\epsilon_i}{|i\epsilon_i|}=\eta\in Lie(G_\mathbf K)$. Denote $c'=min\{\langle w_\alpha,\eta\rangle|\alpha\in I\}$ and $I'=\{\alpha|\langle w_\alpha,\eta\rangle=c', \alpha\in I\}$. Write $i\epsilon_i=|i\epsilon_i|(\eta+\eta_i)$ with $\eta_i\rightarrow 0$. Then by (\ref{l1}) we have
\begin{align}\lim_{i\rightarrow \infty}[e^{-i\langle w_\alpha, \lambda+\epsilon_i \rangle}z_\alpha]
&=\lim_{i\rightarrow \infty}[e^{-\langle w_\alpha, |i\epsilon_i|(\eta+\eta_i) \rangle}z_\alpha(\alpha\in I),0]\notag\\
&=\lim_{i\rightarrow \infty}[e^{-\langle w_\alpha, |i\epsilon_i|\eta_i \rangle}z_\alpha(\alpha\in I'),0].
\notag
\end{align}
The right hand side has the same form as that of (\ref{l1}) with a smaller set $I'$. If the sequence $|i\epsilon_i|\eta_i$ is bounded, we are done. Otherwise, as above, we can assume that $\lim_{i\rightarrow \infty}\frac{|i\epsilon_i|\eta_i}{|i\epsilon_i||\eta_i|}=\eta_1\in Lie(G_\mathbf K)$ by subsequence. Continuing this process, we will get a sequence of vectors $\eta, \eta_1,...,\eta_l$ in $Lie(G_\mathbf K)$ for some $l\leq |I|$. Now for small positive number $s,s_1,...,s_l$, denoting $\tilde v=diag\{\lambda^k+s\eta^k+...+s_l\eta^k_l\}$,
 we have 
$$[X_\infty]=C\cdot\lim_{i\rightarrow \infty}e^{i\tilde v}[X],$$
where $C$ is some element in $\bar G_{\mathbf K}$. Since $e^{\tilde v}\in G_{\mathbf K}$, it generates a $\mathbf G-$equivariant 1-PS degeneration. By Theorem \ref{special-classification}, there is a $\tilde \Lambda \in \mathcal V_{\mathbb R}$ corresponding to this degeneration. The lemma is proved.
\end{proof}

\begin{proof}[Completion of proof of Theorem \ref{main}]
By Theorem \ref{th1},  the    soliton VF  $v$ of  $(X_\infty,  \omega_\infty)$ induces  a special torus $1$-PS degeneration  $e^{tv}$ of $X$  with  central fiber $\bar X_\infty$.  
Since $\bar X_\infty$ can also be a central fiber of torus $\mathbb C^*$-degeneration by Proposition \ref{RTC-appr},   
 $\bar X_\infty$ is  a $\mathbf G$-spherical variety  by  \cite[Theorem 3.30] {Del1}.
 On the other hand,   the  degeneration $(X, e^{tv})$  attains the minimum of H-invariant $H(\mathcal X^\sigma)$  for special $1$-PS degenerations of $X$  by  \cite[Theorem 3.2] {DS16} (also see  \cite[Proposition 5.6] {WZ}),  we may  restrict  the H-invariant for  special  torus $1$-PS degenerations.  By Theorem \ref{special-classification} and Proposition \ref{hinv}, there is a  unique $\Lambda_v \in \mathcal V_{\mathbb R}$  as a minimizer of  $H(\mathcal X^\sigma)$ 
 in (\ref{Hfor}) which corresponds to the torus $1$-PS degeneration $(X, e^{tv})$. Thus by perturbing   the soliton VF  $v$ to $\tilde v$  in  Lemma \ref{uni},   there is another  element $\tilde \Lambda \in \mathcal V_{\mathbb R}$  which induces a   special  torus $1$-PS degeneration $e^{i \tilde v}$  from $X$ to $X_\infty$. Again by  Proposition \ref{RTC-appr} together with Theorem \ref{special-classification},  we may choose $\tilde \Lambda \in \mathcal V_{\mathbb R}$ such that the degeneration  $(X, e^{i \tilde v})$  induces  a  special  torus $\mathbb C^*$-degeneration  of $X$ with  the central fiber $ X_\infty$.  By \cite[Theorem 3.30] {Del1},   it follows that $ X_\infty$ is also a $\mathbf G$-spherical variety.  
 Therefore, the  proof  is completed.

 \end{proof}

%\newpage


\begin{thebibliography}{10}

\bibitem{Al} Alexeev, V. and Brion, M.,
\textit{Stable reductive varieties \uppercase\expandafter{\romannumeral 1}: Affine varieties},
Invent. Math. \textbf{157} (2004), 227-274.


\bibitem{AB2} Alexeev, V.A. and Brion, M., \textit{Stable reductive varieties \uppercase\expandafter{\romannumeral 2}: Projective case}, Adv. Math., \textbf{184} (2004), 382-408. %https://doi.org/10.1016/S0001-8708(03)00164-6

\bibitem{AK} Alexeev, V.A. and Katzarkov, L.V.,  \textit{On K-stability of reductive varieties}, Geom. Funct. Anal., \textbf{15} (2005), 297-310. %https://doi.org/10.1007/s00039-005-0507-x




    \bibitem {Bam} Bamler, R., \textit{Convergence of Ricci flows with bounded scalar curvature},  Ann. Math., \textbf {188}  (2018), 753-831.

 
 \bibitem{BBEGZ} Berman, R., Boucksom S., Essydieux, P., Guedj, V. and Zeriahi A.,  \textit{K\"{a}hler-Einstein metrics and the K\"{a}hler-Ricci flow on log Fano varieties},  arXiv:1111.7158;    \emph{ J. Reine Angew. Math.}, \textbf{751} (2019),  27-89.



  \bibitem{BW} Berman, R.  and Witt-Nystr\"om, D.,   \textit{Complex optimal transport and  the pluripotential  theory of  K\"ahler-Ricci solitons},   arXiv:1401.8264.


 \bibitem{BLXZ}Blum, H., Liu, Y.C., Xu, C.Y. and Zhuang, Z.Q., \textit{The existence of the Kähler–Ricci soliton degeneration}, Forum of Mathematics, Pi, \textbf {11}, E9. doi:10.1017/fmp.2023.5

  \bibitem {Bri89}Brion,  M.,  \textit{Groupe de Picard et nombres caractéristiques des variétés sphériques}, Duke
Math. J., 58(2):397–424, 1989.




 \bibitem{cao} Cao, H.D., \textit{Deformation of K\"ahler metrics to
 K\"ahler-Einstein metrics on compact K\"ahler manifolds}, Invent.
 Math., \textbf{81} (1985), 359-372.



  \bibitem{CSW} Chen, X.X., Sun, S., and Wang B.,  \textit{K\"ahler-Ricci flow, K\"ahler-Einstein metric, and K-stability},
  Geometry and Topology, \textbf{22} (2018), 3145-3173.

    \bibitem{Chwang}Chen, X.X. and Wang, B.,
    \textit{Space of Ricci flows (\uppercase\expandafter{\romannumeral 2})--Part B: Weak compactness of the flows}, J. Diff. Geom., \textbf{116} (2020), 1-123.


\bibitem{DatS}   Datar, D.  and   Sz\'ekelyhidi, G.,    \textit{K\"ahler-Einstein metrics along the smooth continuity
  method}, Geom. Funct. Anal. 26 (2016), no. 4, 975-1010.

\bibitem{DY} Darvas, T. and Rubinstein, Y.,   \textit{Tian's properness conjecture and Finsler geometry of the space of Kahler metrics},   JAMS,  30 (2018) , n2,  347-387. 

\bibitem{Del1}Delcroix, T., \textit{K-Stability of Fano spherical varieties},  Ann Sci \'{E}c Norm Sup\'{e}r., \textbf{53} (2020),  615-662.

\bibitem{Del2} Delcroix, T.,
\textit{K\"ahler geometry of   horosymmetric  varieties, and application to Mabuchi's K-energy functional},  \emph{ J Reine Angew Math.}, \textbf{763} (2020),  129-199.


\bibitem{Del-2020} Delcroix, T., \textit{Uniform K-stability of polarized spherical varieties}, to appear in Épijournal Géom. Algébrique, arXiv:2009.06463.

\bibitem{DH} Delcroix, T. and   Hultgren, J.,
\textit{Coupled complex Monge-Amp\`ere equations on Fano horosymmetric manifolds},  \emph{J. Math. Pures Appl.},  \textbf{153} (2021), 281-315.


  \bibitem{DS16} Dervan, R. and Sz\'ekelyhidi G.,   \textit{K\"ahler-Ricci flow and optimal degenerations}, J. Diff. Geom., \textbf{116} (2020), 187-203.

\bibitem{DT}
Ding, W.Y. and Tian, G.,   K\"{a}hler-Einstein metrics and the generalized
Futaki invariant. Invent. Math., 110 (1992), no. 2, 315-335.

\bibitem{Do} Donaldson, S.K., \textit{Scalar curvature and stability of toric varieties}, J. Differential Geom., \textbf{62} (2002) 289-348.



 \bibitem{GH15}Gagliardi G. and Hofscheier J., \textit{Gorenstein spherical Fano varieties}, Geom.
Dedicata, \textbf{178} (2015),111-133.


\bibitem{HL}Han, J.Y. and Li, C., \textit{Algebraic uniqueness of K\"ahler-Ricci flow limits and optimal degenerations of Fano varieties}, arXiv:2009.01010.

\bibitem{HL2} Han, J.Y. and Li, C., \textit{On the Yau-Tian-Donaldson conjecture for generalized K\"ahler-Ricci soliton equations}, arXiv:2006.00903.


\bibitem{K}Knop, F., \textit{The Luna-Vust theory of spherical embeddings}, In Proceedings of the
Hyderabad conference on algebraic groups held at the School of Mathematics and Computer/Information Sciences of the University of Hyderabad, India, December 1989, 225-249. Madras: Manoj Prakashan, 1991.

 \bibitem {LL}Li, Y. and Li, Z.Y., \textit{Equivariant $\mathbb R$-test configurations and semistable limits of $\mathbb Q$-Fano group compactifications}, to appear in Peking Math. J.
 
\bibitem {LL2}Li, Y. and Li, Z.Y., \textit{Equivariant $\mathbb R$-test configurations of polarized spherical varieties}, arXiv:2206.04880.
  
   \bibitem {LTZ1}  Li, Y., Tian G. and Zhu, X.H., Singular limits of K\"ahler-Ricci flow on Fano G-manifolds, arXiv:1807.09167,
to appear in Amer. J. Math.

\bibitem{LTZ2}   Li, Y., Tian, G. and Zhu, X. H.,  \textit{Singular K\"ahler-Einstein metrics on  $\mathbb Q$-Fano compactifications of a Lie group}, \emph{Math. Eng.},   \textbf{5} (2023), no. 2,  Paper No. 028, 43pp ;  arXiv: 2001. 11320.

\bibitem{LZZ}  Li, Y.,  Zhou, B. and Zhu, X. H., \textit{ K-energy on polarized compactifications of Lie
groups},  \emph{J. of Func. Analysis},  \textbf{275} (2018), 1023-1072.


 \bibitem {LLW} Li, Y., Li, Z.Y. and Wang, F., \textit{Weighted K-stability of $\mathbb Q$-Fano spherical varieties},  arXiv:2208.02708.

 \bibitem {LWX} Li, C., Wang, X.W. and Xu, C.Y., \textit{Algebraicity of the metric tangent cones and equivariant K-stability}, J. Amer. Math. Soc. \textbf{34}(2021), 1175-1214.


    \bibitem {LS}  Liu, G. and Sz\'ekelyhidi, G.,  \textit{  Gromov-Hausdorff limit of  K\"{a}hler manifolds with Ricci bounded below}, Geom. Funct. Anal. \textbf{32} (2022), 236–279. https://doi.org/10.1007/s00039-022-00594-8


%\bibitem{LLW}   Li, Y. ,  Li, Z. and Wang, F.,  \textit{ Weighted $K$-stability  of $\mathbb Q$-Fano  spherical varieties},  arXiv:2208.02708.



 \bibitem {RT}  Ross, J. and Thomas, R., \textit{  A study of the Hilbert-Mumford criterion for the stability of projective
varieties}, J. Algebraic Geom. \textbf{16}(2007), 201–255.

 \bibitem {TS1} Tian, G. and Song, J., \textit{ The K\"ahler–Ricci flow through singularities}, Invent. math. 207, 519–595 (2017). https://doi.org/10.1007/s00222-016-0674-4




    \bibitem {Ti97} Tian, G.,
    \textit{K\"ahler-Einstein metrics with positive scalar curvature},
    Invent. Math., \textbf{130} (1997), 1-37.

 \bibitem {TZ2} Tian, G. and Zhu, X. H.,
    \textit{A new holomorphic invariant and uniqueness of K\"ahler-Ricci solitons}, Comment. Math. Helv., \textbf{77} (2002), 297-325.

 \bibitem {TZ3}Tian, G. and Zhu, X. H.,  \textit{Horosymmetric limits of K\"ahler-Ricci flow on Fano $G-$manifolds}, arXiv:2209.05029.


   \bibitem{TZhzh} Tian, G. and Zhang, Z.L.,
    \textit{Regularity of K\"ahler-Ricci flows on Fano manifolds}, Acta Math., \textbf{216} (2016), 127-176.

     \bibitem{TZZZ} Tian, G.,  Zhang, S.J., Zhang, Z.L. and Zhu, X.H.,  \textit{Perelman's entropy and K\"ahler-Ricci flow an a Fano Manifold}, Tran. AMS., \textbf{365} (2013), 6669-6695.


  \bibitem{Tim}Timash\"ev D.A., \textit{Homogenous spaces and equivariant embeddings}, Encyclopaedia of Mathematical Sciences, \textbf{138}. Invariant Theory and Algebraic Transformation Groups, \textbf{8}. Springer Verlag, Berlin-Heidelberg, 2011.

   \bibitem{WZ20}Wang, F. and Zhu, X.H., \textit{Tian's partial $C^0$-estimate implies Hamilton-Tian's conjecture},  Adv. Math., \textbf{381} (2021), Paper No. 107619. https://doi.org/10.1016/j.aim.2021.107619

   \bibitem{WZ}Wang, F. and Zhu, X.H., \textit{Uniformly strong convergence of K\"ahler-Ricci flows  on a Fano manifold}, Sci. China Math., (2022). https://doi.org/10.1007/s11425-021-1928-1

     \bibitem{WZZ} Wang, F., Zhou, B. and Zhu, X.H., \textit{Modified Futaki invariant and equivariant Riemann-Roch formula}, Adv. Math., \textbf{289} (2016), 1205-1235.
     
     \bibitem{WZ04} Wang, X.   and  Zhu, X.H., 
\textit{K\"ahler-Ricci solitons on toric manifolds with positive first Chern class},
Adv. Math. \textbf{188} (2004), 87-103.


 \bibitem{Xio} Xiong, M.,  \textit{K\"ahler-Ricci solitons and generalized Tian-Zhu's invariant},
Internat. J. Math.  \textbf{25} (2014).


% \bibitem {XuZ} Xu, C.Y. and Zhuang, Z.
 
 
   \bibitem{ZhK} Zhang, K.,  \textit{Some refinements of the partial $C^0$-estimate}, Anal. PDE \textbf{14} (2021), 2307-2326. https://doi.org/10.2140/apde.2021.14.2307

\bibitem{Zhu}Zhuang, Z. \textit{Optimal destabilizing centers and equivariant K-stability}, Invent. math. \textbf{226},(2021),195–223. https://doi.org/10.1007/s00222-021-01046-0

   %\bibitem {Zh} Zhang, Q. \textit{A uniform Sobolev inequality under Ricci flow}, Int. Math. Res. Notices, \textbf{17} (2007), 1-17.

   %\bibitem{Zhq} Zhang, Q. \textit{Bounds on volume growth of geodesic balls under Ricci flow}, Math. Res. Lett.,
   %\textbf{19} (2012), 245-253.


   %\bibitem{ZhK} Zhang, K. \textit{ On the partial $C^0$-estimate},  arXiv: 1911.11328.


   \end{thebibliography}
\end{document}